\documentclass[10pt,a4paper,twoside]{article}
\usepackage{amssymb,amsmath,amsfonts,amsthm,graphicx}

 \newtheorem{prop}{Proposition}[section]
 \newtheorem{lemma}{Lemma}[section]
 \newtheorem{thm}{Theorem}[section]
 \theoremstyle{definition}
 \newtheorem{definition}{Definition}[section]
 \theoremstyle{remark}
 \newtheorem{remark}{Remark}

 \newcommand{\E}{\mathbb{E}}
 \newcommand{\N}{\mathbb{N}}
 \newcommand{\R}{\mathbb{R}}
 \newcommand{\T}{\mathbb{T}}
 
 \newcommand{\W}{\boldsymbol{\mathcal{W}}}
 \newcommand{\Z}{\mathbb{Z}}

 \renewcommand{\geq}{\geqslant}
 \renewcommand{\leq}{\leqslant}

 \numberwithin{equation}{section}
 \numberwithin{figure}{section}
 \long\def\symbolfootnote[#1]#2{\begingroup%
 \def\thefootnote{\fnsymbol{footnote}}\footnote[#1]{#2}\endgroup}
\begin{document}

\thispagestyle{empty}

 \vskip2cm
 {\centering
 \Large{\bf Explicit Bivariate Rate Functions for Large Deviations
in AR(1) and MA(1) Processes with Gaussian Innovations}\\
 }
 \vspace{1.0cm}
 \centerline{\large{\bf{M.J. Karling, A.O. Lopes  and S.R.C. Lopes
\symbolfootnote[3]{Corresponding author. E-mail: silviarc.lopes@gmail.com}}}}

 \vspace{0.5cm}

 \centerline{Mathematics and Statistics Institute}
 \centerline{Federal University of Rio Grande do Sul}
 \centerline{Porto Alegre, RS, Brazil}

 \vspace{0.5cm}

 \centerline{\today}

\begin{abstract}
\noindent
 We investigate large deviations properties for centered
stationary AR(1) and MA(1) processes with independent Gaussian innovations,
by giving the explicit bivariate rate functions for the sequence of random
vectors $(\boldsymbol{S}_n)_{n \in \N} = \left(n^{-1}(\sum_{k=1}^n X_k,
\sum_{k=1}^n X_k^2)\right)_{n \in \N}$. In the AR(1) case, we also give the
explicit rate function for the bivariate random sequence $(\W_n)_{n \geq 2} =
\left(n^{-1}(\sum_{k=1}^n X_k^2, \sum_{k=2}^n X_k X_{k+1})\right)_{n \geq
 2}$. Via Contraction Principle, we provide explicit rate functions for
the sequences  $(n^{-1} \sum_{k=1}^n X_k)_{n \in \N}$,
$(n^{-1} \sum_{k=1}^n X_k^2)_{n \geq 2}$ and $(n^{-1} \sum_{k=2}^n X_k
X_{k+1})_{n \geq 2}$, as well. In the AR(1) case, we present a new proof for
an already known result on the explicit deviation function for the
Yule-Walker estimator.
\medskip

\medskip

\end{abstract}

\section{Introduction}

Since the first establishments on the Large Deviations theory, there has been
a great expansion of the number of surveys on Large Deviations
Principles (LDP). Nowadays, we can find a variety of examples applied to the
time series analysis and stochastic processes in general; for instance,
LDPs for Stable laws (see, e.g. Heyde \cite{heyde67}, Rozovskii
\cite{rozovskii89},
Rozovskii \cite{rozovskii99} and Zaigraev \cite{zaigraev99}), stationary
Gaussian processes
(see, e.g. Bercu {\it et al}.\ \cite{bercu00}, Bercu {\it et al}.\
\cite{bercu97}, Bryc and Dembo
\cite{bryc97}, Donsker and Varadhan \cite{donsker85} and
Zani \cite{zani13}), autoregressive and moving average processes
(see, e.g. Bercu \cite{bercu01}, Bryc and Smolenski \cite{bryc93}, Burton and
Dehling \cite{burton90}, Djellout and Guillin
\cite{djellout01}, Macci and Trapani \cite{macci13}, Mas and Menneteau
\cite{mas03}, Miao \cite{miao09} and Wu \cite{wu04})
and continuous processes (see, e.g. Bercu and Richou \cite{bercu15} and
Bercu and Richou \cite{bercu17}).

When considering the \emph{Empirical Autocovariance} function
\begin{equation*}
\tilde\gamma_n(h) = \frac{1}{n} \sum_{k=1}^{n-h} X_k X_{k+h}, \quad  \mbox{for }
0 \leq h \leq d \quad \mbox{and} \quad d \in \N,
\end{equation*}
of a process $(X_n)_{n \in \N}$, few results on LDP are known. Regarding
Gaussian
distributions, one of the first studies in the literature is the one from
Bryc and Smolenski \cite{bryc93}, concerning the LDP for the
\emph{Quadratic Mean}
  \begin{equation*}
   \tilde\gamma_n(0) = \frac{1}{n} \sum_{k=1}^n X_k^2.
  \end{equation*}
Bryc and Dembo \cite{bryc97} showed that an LDP for the
vector $({\tilde\gamma}_{n}(h))_{h=0}^d$ is available when $(X_n)_{n \in \N}$ is
an
independent and identically distributed (i.i.d.)\ process, with $X_n \sim
\mathcal{N}(0,1)$. It is well known that most of the relevant stochastic
processes are not independent and, as the authors have claimed, their approach
needs some adjustments when trying to show that a similar LDP works, for
instance, when dealing with the classical centered stationary Gaussian
AR(1) process (see example 1 in Bryc and Dembo \cite{bryc97}). On the other
hand, Bercu {\it et al}.\ \cite{bercu97} proved an LDP for Toeplitz quadratic
forms of centered stationary Gaussian processes in an univariate setting.
Their
survey eliminated the need for the variables of $(X_n)_{n \in \N}$ to be
independent,
extending the result in Bryc and Dembo \cite{bryc97} by including the AR(1)
process. However, it is not clear if the LDP was available even for the
bivariate random vector
$(\tilde\gamma_{n}(0), \tilde\gamma_{n}(1))$, once the LDP has only been proved
for each one of the components separately. More precisely, the results in
Bercu {\it et al}.\ \cite{bercu97} only cover the LDP of the random variable
\begin{equation*}
 W_n = \frac{1}{n} {X^{(n)}}^T M_n\,X^{(n)},
\end{equation*}
where $X^{(n)} = (X_1, \cdots, X_n)$, with ${X^{(n)}}^T$ denoting the
transpose of $X^{(n)}$, and where $(M_n)_{n \in \N}$ is a sequence of $n \times
n$ Hermitian matrices.

In a more general setting, Carmona {\it et al}.\ \cite{carmona98} present a
level-1 LDP for the empirical autocovariance function of order $h$ for
any innovation processes, that encompasses the AR($d$) process with Gaussian
innovations. In this paper, the authors used the level-2 LDP together with the
Contraction Principle. The process itself is obtained from iterations of a
continuous uniquely ergodic transformation, preserving the Lebesgue measure on
the circle. In Carmona and Lopes \cite{carmona00}, the authors considered a
similar problem where the dynamics are given by an expanding transformation on
the circle. In the same line of research, Wu \cite{wu04} proved an LDP for
$(\tilde\gamma_{n}(h))_{h=0}^d$ under the assumption that $\E(\exp(\lambda
\varepsilon_n^2))$ is finite, for $\lambda > 0$, where $(\varepsilon_n)_{n
\in \N}$ is
the white noise of an AR($d$) process, excluding in turn the Gaussian case.

In the present manuscript, we take into account the studies from
Bercu {\it et al}.\ \cite{bercu97} and Bryc and Dembo \cite{bryc97} to give a
proof that the sequence $(\W_n)_{n
\geq
2}$, given by
\begin{equation*}
  \W_n = \frac{1}{n} \left(\tilde\gamma_n(0),\tilde\gamma_n(1)\right) =
\frac{1}{n}
\left(\sum_{k=1}^n X_k^2, \sum_{k=2}^n X_k X_{k-1}\right), \quad
\mbox{for } n \geq 2,
\end{equation*}
does, in fact, satisfy an LDP when $(X_n)_{n \in \N}$ is a centered stationary
Gaussian AR(1) process and we present its explicit bivariate rate
function. The asymptotical behavior of the sequence
$(\W_n)_{n \geq 2}$ is well known (see Brockwell and Davis \cite{brockwell91}),
that is
 \begin{equation*}
  \W_n \xrightarrow{n \to \infty} \left(\frac{1}{1 - \theta^2},
\frac{\theta }{1 - \theta^2}\right), \quad \mbox{almost surely}.
 \end{equation*}
By definition of almost sure convergence, as $n \to \infty$, the sequence
of probabilities
\begin{equation}\label{prob}
\mathbb{P}\left(\ \left|\left|\W_n - \left(\frac{1}{1-\theta^2} ,
\frac{\theta}{1-\theta^2}\right)\right|\right| > \delta\right),
\end{equation}
converges to zero, for all $\delta > 0$. However, if the
convergence of these probabilities is very slow, even for large $n$, we have a
certain reasonable chance of choosing a bad sample $X_1, \cdots, X_n$ from
$(X_n)_{n \in \N}$, such that $\W_n$ is distant from the true
value $\left(\frac{1}{1-\theta^2} , \frac{\theta}{1-\theta^2}\right)$.

The Large Deviations theory considers the asymptotic behavior of the
probabilities presented in \eqref{prob}, ensuring that they
converge to zero approximately in exponential rate (see chapter 1 in
Bucklew \cite{bucklew90}). Its usual definition is given as follows (see Dembo
and Zeitouni \cite{dembo10}).

 \begin{definition}\label{drate}
  A sequence of random vectors $(\boldsymbol{V}_n)_{n \in \N}$ of $\R^d$, for $d
\in
\N$,
satisfies a \emph{Large Deviation Principle} (LDP) with speed $n$ and
\emph{rate function} $J(\boldsymbol{\cdot})$, if $J(\boldsymbol{\cdot}):\R^d
\rightarrow
[0,
\infty]$ is a
lower semi-continuous function such that,
  \begin{itemize}
   \item Upper bound: for any closed set $F \subset \R^d$,
   \begin{equation*}
    \limsup_{n \to \infty} \frac{1}{n} \log \mathbb{P}\left(\boldsymbol{V}_n \in
F\right) \leq -\inf_{\boldsymbol{x} \in F} J(\boldsymbol{x});
   \end{equation*}
   \item Lower bound: for any open set $G \subset \R^d$,
   \begin{equation*}
    -\inf_{\boldsymbol{x} \in G} J(\boldsymbol{x}) \leq \liminf_{n \to \infty}
\frac{1}{n} \log \mathbb{P}\left(\boldsymbol{V}_n \in F\right).
   \end{equation*}
  \end{itemize}
  Moreover, $J(\boldsymbol{\cdot})$ is said to be a \emph{good rate function} if
its
level sets $J^{-1}([0,b])$ are compact, for all $b \in \R$.
 \end{definition}

\begin{remark}
 In this work, we only deal with good rate functions, but for short, we
sometimes write rate function instead.
\end{remark}

In general, it is not easy to prove that an arbitrary sequence of random
vectors satisfies an LDP (see, e.g.
Bercu and Richou \cite{bercu17}, Bryc and Dembo \cite{bryc97}, Dembo and
Zeitouni \cite{dembo10}, Ellis \cite{ellis85}, Macci and Trapani
\cite{macci13} and Mas and Menneteau \cite{mas03}). An elegant way of proving
such
property is to verify the validity of the G\"{a}rtner-Ellis' theorem conditions
 (see theorem 2.3.6 in Dembo and Zeitouni \cite{dembo10}), which is a
counterpart to the
very well known Cram\'{e}r-Chernoff's theorem (see theorem 2.2.30 in
Dembo and Zeitouni \cite{dembo10}). It is worth mentioning that,
within the conditions of the G\"{a}rtner-Ellis' theorem, little use of the
dependency structure is made and the focus mainly rests in the behavior of the
\emph{limiting cumulant generating function}, defined by
 \begin{equation*}
  L(\boldsymbol{\lambda}) = \lim_{n \to \infty} L_n(\boldsymbol{\lambda}),
\mbox{ for
all } \boldsymbol{\lambda} \in \R^2,
 \end{equation*}
 where $L_n(\boldsymbol{\cdot}):\R^2 \rightarrow \R \cup \{\infty\}$  denotes
the
\emph{normalized cumulant generating function} of $\W_n$,
 \begin{equation*}
  L_n(\boldsymbol{\boldsymbol{\lambda}}) = \frac{1}{n} \log
\mathbb{E}\left[\exp\left(n
\langle\boldsymbol{\lambda}, \W_n\rangle \right)\right].
 \end{equation*}
We shall present an explicit expression for $L(\boldsymbol{\cdot})$ in the case
$\boldsymbol{\lambda} = (\lambda_1, \lambda_2)$ depends on two variables. As a
result,
we obtain the explicit rate function through the Fenchel-Legendre transform of
$L(\boldsymbol{\cdot})$.

In the second part of our study, we shall analyze the LDP of the sequence of
bivariate random vectors $(\boldsymbol{S}_n)_{n \in \N}$, where
\begin{equation*}
 \boldsymbol{S}_n = \frac{1}{n} \left(\sum_{k=1}^n X_k, \sum_{k=1}^n
X_k^2\right).
\end{equation*}
We shall call $\boldsymbol{S}_n$ as the \emph{bivariate SQ-Mean}, for
short, since its first and second components are, respectively, the Sample Mean
and the Quadratic Mean. We dedicate our study to the particular cases
when $(X_n)_{n \in \N}$ follows an AR(1) or an MA(1) process. This study is
based on a particular result presented in Bryc and Dembo \cite{bryc97} and
which has a very
interesting application when the Contraction Principle can be applied.

Our study is organized as follows. Section 2 is dedicated to the
proof of the LDP and computations of the explicit rate function for the random
sequence $(\W_n)_{n \geq 2}$, under the assumption that $(X_n)_{n \in \N}$
follows an AR(1) process. In Section 3, we obtain the LDP for some particular
cases, namely, the Quadratic Mean and the first order Empirical Autocovariance
of a random sample $X_1, \cdots, X_n$ from the AR(1) process. Moreover, the LDP
for the Yule-Walker estimator is provided likewise. As a direct
application of the studies in Section 2, we dedicate Section 4 to show that the
LDP for the SQ-Mean $(\boldsymbol{S}_n)_{n \in \N}$ of an AR(1) process is
available.
Next, we give the details of the LDP for the Quadratic Mean of an MA(1) process
and, as a consequence, the LDP for the SQ-Mean. Section 5 gives insights on
future work and concludes the manuscript.

\section{LDP and the centered stationary Gaussian AR(1) process}

Consider the autoregressive process $(X_n)_{n \in \N}$ defined by the
equation
 \begin{equation}\label{ar1}
  X_{n+1} = \theta X_n + \varepsilon_{n+1}, \quad \mbox{ for } |\theta| < 1
\mbox{ and } n \in \N,
 \end{equation}
where $(\varepsilon_n)_{n \geq 2}$ is a sequence of i.i.d.\ random variables,
with $\varepsilon_n
\sim \mathcal{N}(0, 1)$, for all $n \geq 2$. Assume that $X_1$ is
independent of $(\varepsilon_n)_{n \geq 2}$, with $\mathcal{N}(0,
1/(1-\theta^2))$ distribution. Then $(X_n)_{n\in\N}$ is a centered stationary
Gaussian AR(1) process with (positive) spectral density function defined as
\begin{equation}\label{spectral}
  g_\theta(\omega) = \frac{1}{1 + \theta^2 - 2 \theta \cos(\omega)}, \quad
\omega \in \T = [-\pi,\pi).
 \end{equation}

Throughout this section, we shall study the existence of an LDP for the
random vector
\begin{equation}\label{Wn}
 \W_n = \frac{1}{n} \left(\sum_{k=1}^n X_k^2, \sum_{k=2}^n X_k X_{k-1}\right).
\end{equation}

Consider $\boldsymbol{\lambda} = (\lambda_1, \lambda_2) \in\R^2$. Let
$L_n(\cdot,\cdot): \R^2 \rightarrow \R$ represent the normalized cumulant
generating function associated to the sequence $(\W_n)_{n \geq 2}$, defined by
\begin{equation}\label{Ln1}
  L_n(\lambda_1, \lambda_2) = \frac{1}{n} \log \E \left( e^{n \langle
(\lambda_1, \lambda_2), \W_n \rangle} \right), \quad \mbox{for } n \geq 2,
\end{equation}
where $\langle (x_1, y_1), (x_2, y_2) \rangle := x_1 x_2 + y_1 y_2$ denotes
the usual inner product in $\R^2$. We want to apply the
G\"{a}rtner-Ellis' theorem, which requires the convergence of
$L_n(\cdot, \cdot)$, as $n \to \infty$.

\subsection{Analysis of the normalized cumulant generating
function}\label{analysisNCG}
We shall present below, the expression for the limiting function
$L(\cdot,\cdot)$, when $n \to \infty$, of the  sequence of functions
$(L_n(\cdot,\cdot))_{n \geq 2}$. In particular, we use the function
$L(\cdot,\cdot)$ by applying the G\"{a}rtner
Ellis' theorem in order to obtain the rate function of the sequence $(\W_n)_{n
\geq 2}$.

With $X^{(n)} = (X_1, \cdots, X_n)$ and ${X^{(n)}}^T$ denoting the
transpose of $X^{(n)}$, note that, one can rewrite \eqref{Wn} as
  \begin{equation}\label{Wnphi}
   \W_n = \frac{1}{n} \left( {X^{(n)}}^T T_n(\varphi_1) X^{(n)}, \ {X^{(n)}}^T
T_n(\varphi_2) X^{(n)} \right),
  \end{equation}
  where $\varphi_1: \T \rightarrow \{1\}$ and $\varphi_2: \T
\rightarrow
[-1,1]$ are real valued functions, given respectively by
  \begin{equation*}
   \varphi_1(\omega) = 1 \qquad \mbox{and} \qquad \varphi_2(\omega) =
\cos(\omega).
  \end{equation*}
  The matrix $T_n(f)$ represents the Toeplitz matrix associated
to the function $f: \T \rightarrow \R$, which is defined by
  \begin{equation*}
   T_n(f) = \left[\frac{1}{2\pi} \int_\T e^{i\,(j-k) \omega} f(\omega) \ d
\omega\right]_{1 \leq j, k \leq n}.
  \end{equation*}

\begin{remark}
 A vast literature comprehending Toeplitz matrices has emerged in the last
century and one of the most famous and referenced works is given in Grenander
and Szeg\"{o} \cite{grenander58}. A modern treatment about this subject may be
found in Gray \cite{gray06} and in Nikolski \cite{nikolski20}.
\end{remark}

  Inserting \eqref{Wnphi} into \eqref{Ln1}, we obtain
  \begin{equation*}
   L_n(\lambda_1, \lambda_2) = \frac{1}{n} \log \E \left( e^{{X^{(n)}}^T
(\lambda_1 T_n(\varphi_1) + \lambda_2 T_n(\varphi_2)) X^{(n)}} \right)
  \end{equation*}
  and, by linearity of Toeplitz matrices, we get
  \begin{equation}\label{Ln2}
   L_n(\lambda_1, \lambda_2) = \frac{1}{n} \log \E \left( e^{{X^{(n)}}^T
T_n(\varphi_{\boldsymbol{\lambda}})
X^{(n)}} \right),
  \end{equation}
  where $\varphi_{\boldsymbol{\lambda}}: \T \rightarrow \R$ is defined by
  \begin{equation}\label{philambda}
   \varphi_{\boldsymbol{\lambda}}(\omega) = \lambda_1 \varphi_1(\omega) +
\lambda_2
\varphi_2(\omega) = \lambda_1 + \lambda_2 \cos(\omega).
  \end{equation}
  Observe that $\varphi_{\boldsymbol{\lambda}}(\cdot)$ depends on the choice of
$\boldsymbol{\lambda} = (\lambda_1, \lambda_2)$ and that
  \begin{equation*}
   T_n(\varphi_{\boldsymbol{\lambda}}) =\frac{1}{2} \left[%
   \begin{array}{cccccc}
     2\lambda_1 & \lambda_2 & 0 & 0 & \cdots & 0 \\
     \lambda_2 & 2\lambda_1 & \lambda_2 & 0 & \ddots & \vdots \\
     0 & \lambda_2 & 2 \lambda_1 & \lambda_2 & \ddots & 0 \\
     0 & 0 & \lambda_2 & \ddots & \ddots & 0 \\
     \vdots & \ddots & \ddots & \ddots & 2 \lambda_1 & \lambda_2\\
     0 & \cdots & 0 & 0 & \lambda_2 & 2\lambda_1
   \end{array}\right].
  \end{equation*}

  The fact that $X^{(n)}$ has multivariate Gaussian distribution gives us some
advantage here. A standard result from Probability theory (see section B.6 in
Bickel and Doksum \cite{bickel01}) shows that there is always a multivariate
Gaussian vector
$Y^{(n)} = (Y_{n,1}, \cdots, Y_{n,n})$ with independent components, such that
  \begin{equation}\label{trick}
   X^{(n)} = T_n(g_\theta)^{1/2} \, Y^{(n)},
  \end{equation}
where $g_\theta(\cdot)$ is given in \eqref{spectral}
and $T_n(g_\theta)^{1/2}$ is the square root matrix of $T_n(g_\theta)$. We also
note that $(T_n(g_\theta))_{n \in \N}$ is the sequence of autocovariance
matrices associated to the process $(X_n)_{n \in \N}$. Therefore,
since $T_n(g_\theta)$ is a positive definite matrix, the sequence of matrices
$(T_n(g_\theta)^{1/2})_{n \in \N}$ is well defined.

  From \eqref{trick} we obtain
  \begin{equation}\label{trick2}
   {X^{(n)}}^T  T_n(\varphi_{\boldsymbol{\lambda}})\,X^{(n)} = {Y^{(n)}}^T
T_n(g_\theta)^{1/2}\, T_n(\varphi_{\boldsymbol{\lambda}})\,
T_n(g_\theta)^{1/2}\,Y^{(n)}.
  \end{equation}
  Since $T_n(g_\theta)^{1/2}\, T_n(\varphi_{\boldsymbol{\lambda}})\,
T_n(g_\theta)^{1/2}$ is a real symmetric matrix, there exists a sequence of
orthogonal matrices $(P_n)_{n \in \N}$ such that
  \begin{equation}\label{trick3}
   T_n(g_\theta)^{1/2}\, T_n(\varphi_{\boldsymbol{\lambda}})\,
T_n(g_\theta)^{1/2} =
P_n \,\Lambda_n\, P_n^T,
  \end{equation}
  with $\Lambda_n = \mbox{Diag}(\alpha_{n,1}^{\boldsymbol{\lambda}}, \cdots,
\alpha_{n,n}^{\boldsymbol{\lambda}})$ a diagonal $n \times n$ matrix, where
$(\alpha_{n,k}^{\boldsymbol{\lambda}})_{k=1}^n$ are the eigenvalues of
 \begin{equation*}
  T_n^{1/2}(g_\theta) T_n(\varphi_{\boldsymbol{\lambda}}) T_n^{1/2}(g_\theta).
 \end{equation*}

\begin{remark}
 It is interesting to note that $(\alpha_{n,k}^{\boldsymbol{\lambda}})_{k=1}^n$
are
also
the eigenvalues of $T_n(\varphi_{\boldsymbol{\lambda}})\,T_n(g_\theta)$.
\end{remark}

  From \eqref{trick2} and \eqref{trick3} we obtain
  \begin{equation}\label{trick4}
   {Y^{(n)}}^T T_n(g_\theta)^{1/2}\, T_n(\varphi_{\boldsymbol{\lambda}})\,
T_n(g_\theta)^{1/2}\,Y^{(n)} = {Y^{(n)}}^T P_n \, \Lambda_n \, P_n^T \,Y^{(n)}.
  \end{equation}

  As $P_n$ is orthogonal, the product $P_n^T \, Y^{(n)}$ has a
multivariate Gaussian distribution with independent components. From
\eqref{trick2} and \eqref{trick4}, it is easy to conclude that
  \begin{equation}\label{orth}
   {X^{(n)}}^T  T_n(\varphi_{\boldsymbol{\lambda}})\,X^{(n)} =
\sum_{k=1}^n \alpha_{n,k}^{\boldsymbol{\lambda}} Z_{n,k},
  \end{equation}
  where $Z_{1,n}, \cdots, Z_{n,n}$ are i.i.d.\ random variables, each one
having a $\chi^2_1$ distribution with moment generating function given by
  \begin{equation}\label{mom}
   M_{Z_{n,k}}(t) = \E(e^{t Z_{n,k}}) =
   \begin{cases}
   \displaystyle \frac{1}{(1 - 2 t)^{1/2}},& t < \frac{1}{2},\\
    \hspace{7.5mm} \infty, & t \geq \frac{1}{2},
   \end{cases}
  \end{equation}
  for  $k = 1, \cdots, n$.

  Returning to the analysis of \eqref{Ln2} and considering \eqref{orth}, as
$Z_{1,n}, \cdots, Z_{n,n}$ are mutually independent, we
conclude that
  \begin{equation}\label{Ln3}
   L_n(\lambda_1, \lambda_2) = \frac{1}{n} \log \E \left(e^{\sum_{k=1}^n
\alpha_{n,k}^{\boldsymbol{\lambda}} Z_{n,k}}\right) = \frac{1}{n} \log
\left(\prod_{k=1}^n \E \left( e^{\alpha_{n,k}^{\boldsymbol{\lambda}}
Z_{n,k}}\right)\right).
  \end{equation}
  From \eqref{mom}, we observe that $\E \left(
e^{\alpha_{n,k}^{\boldsymbol{\lambda}} Z_{n,k}}\right)$ is only defined if each
one
of
the $\alpha_{n,k}^{\boldsymbol{\lambda}} < 1/2$. In other words, \eqref{Ln3} is
finite if
  \begin{equation}\label{condalpha}
   0 < 1 - 2 \alpha_{n,k}^{\boldsymbol{\lambda}}, \quad \mbox{for all $k$ such
that}
\ 1 \leq k \leq n.
  \end{equation}

  The condition in \eqref{condalpha} is equivalent to requiring that $I_n
- 2\,T_n(\varphi_{\boldsymbol{\lambda}}) T_n(g_\theta)$ must be positive
definite
(see Bercu {\it et al}.\ \cite{bercu97}). Since $T_n(g_\theta)$ is a positive
definite matrix and
  \begin{equation*}
   I_n - 2\,T_n(\varphi_{\boldsymbol{\lambda}}) T_n(g_\theta) =
(T_n^{-1}(g_\theta)
- 2
T_n(\varphi_{\boldsymbol{\lambda}})) \ T_n(g_\theta),
  \end{equation*}
  it is sufficient to show that
  \begin{equation}\label{Dnlambda}
   D_{n,\boldsymbol{\lambda}} = T_n^{-1}(g_\theta) - 2
T_n(\varphi_{\boldsymbol{\lambda}}) =
\left(\begin{array}{ccccc}
        r_1   & q    & 0      & \cdots & 0\\
        q    & p      & \ddots & \ddots & \vdots\\
        0      & \ddots & \ddots & \ddots & 0\\
        \vdots & \ddots & \ddots & p      & q\\
        0      & \cdots & 0 & q    & r_1
       \end{array} \right)
  \end{equation}
  is positive definite, where $r_1 = 1 - 2 \lambda_1$, $ p = 1 +
\theta^2 - 2\lambda_1$ and $q = - \theta - \lambda_2$. The domain
$\mathcal{D} \subseteq \R^2$, where $D_{n,\boldsymbol{\lambda}}$ (and so $I_n -
2
T_n(\varphi_{\boldsymbol{\lambda}}) T_n(g_\theta)$) is positive definite, is
given
by
the following lemma.

 \begin{lemma}\label{domD}
  If the pair $(\lambda_1, \lambda_2)$ belongs to the domain $\mathcal{D} =
\mathcal{D}_1 \cup \mathcal{D}_2$, where
  \begin{equation}\label{D12}
   \begin{split}
   &\mathcal{D}_1 = \left\{(\lambda_1, \lambda_2) \in \R^2 \, \big| \, \lambda_1
\leq \frac{1 - \theta^2}{2}, \ 4 (\theta + \lambda_2)^2 < (1 + \theta^2 -
2\lambda_1)^2 \right\},\\
   &\mathcal{D}_2 = \left\{(\lambda_1, \lambda_2) \in \R^2 \, \big| \, \frac{1 -
\theta^2}{2} < \lambda_1 < \frac{1}{2}, \ (\theta + \lambda_2)^2 <  \theta^2(1 -
2 \lambda_1)\right\},
   \end{split}
  \end{equation}
  then, for $n$ large enough, the tridiagonal matrix
$D_{n,\boldsymbol{\lambda}}$,
given
in \eqref{Dnlambda}, is positive definite.
 \end{lemma}
 \begin{proof}
  The proof is given in Appendix \ref{appenA}.
 \end{proof}

 To illustrate the domains presented in Lemma \ref{domD}, Figures \ref{domainD1}
and \ref{domainD3} show the graphs of $\mathcal{D}$ when $\theta = 0.9$. In
particular, Figure \ref{domainD1} shows the sets $\mathcal{D}_1$ and
$\mathcal{D}_2$ separately, while Figure \ref{domainD3} shows the union
$\mathcal{D} = \mathcal{D}_1 \cup \mathcal{D}_2$.

 \begin{figure}[ht!]
  \begin{center}
   \includegraphics[scale =0.25]{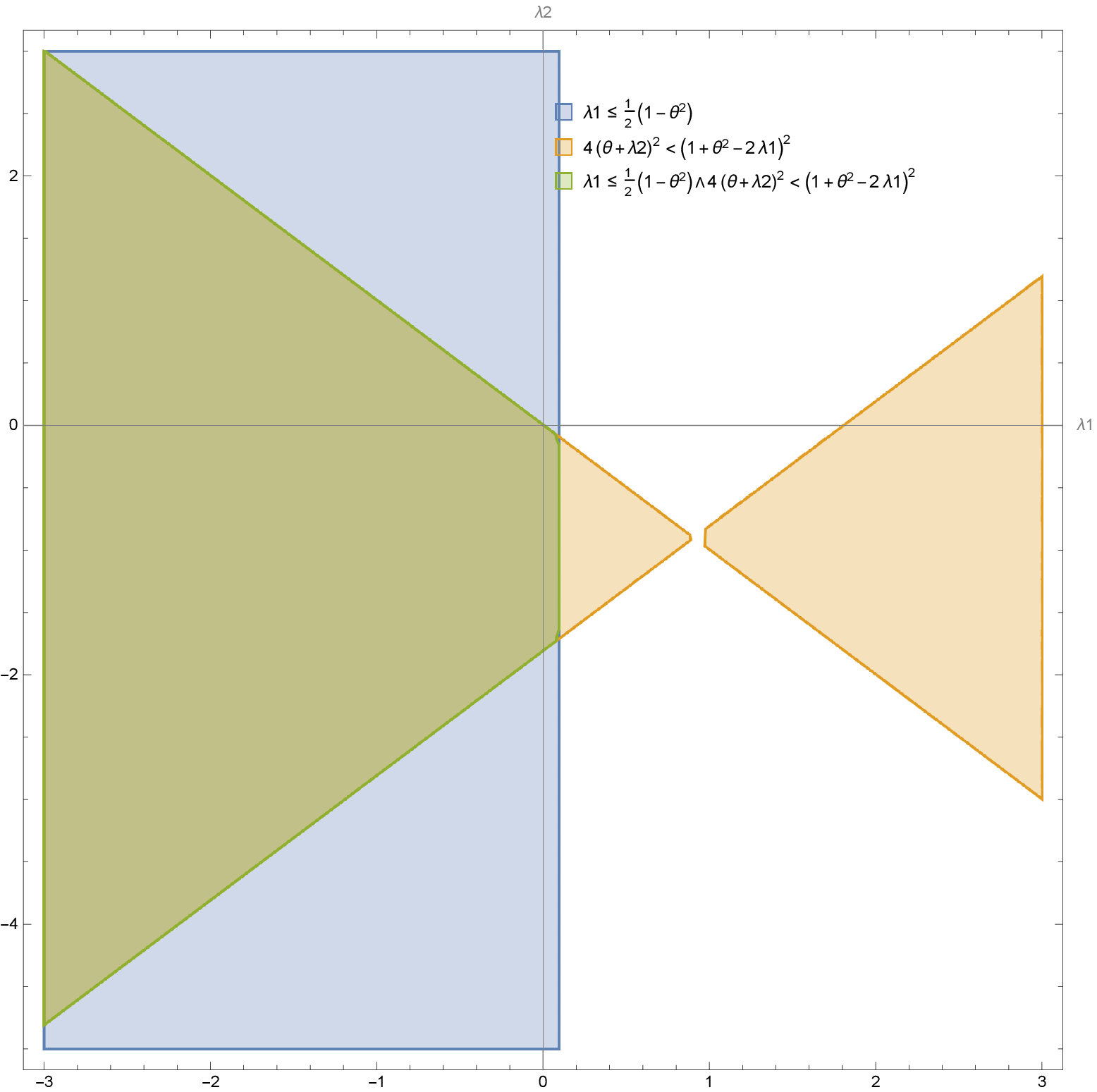} \hspace{15mm}
   \includegraphics[scale =0.25]{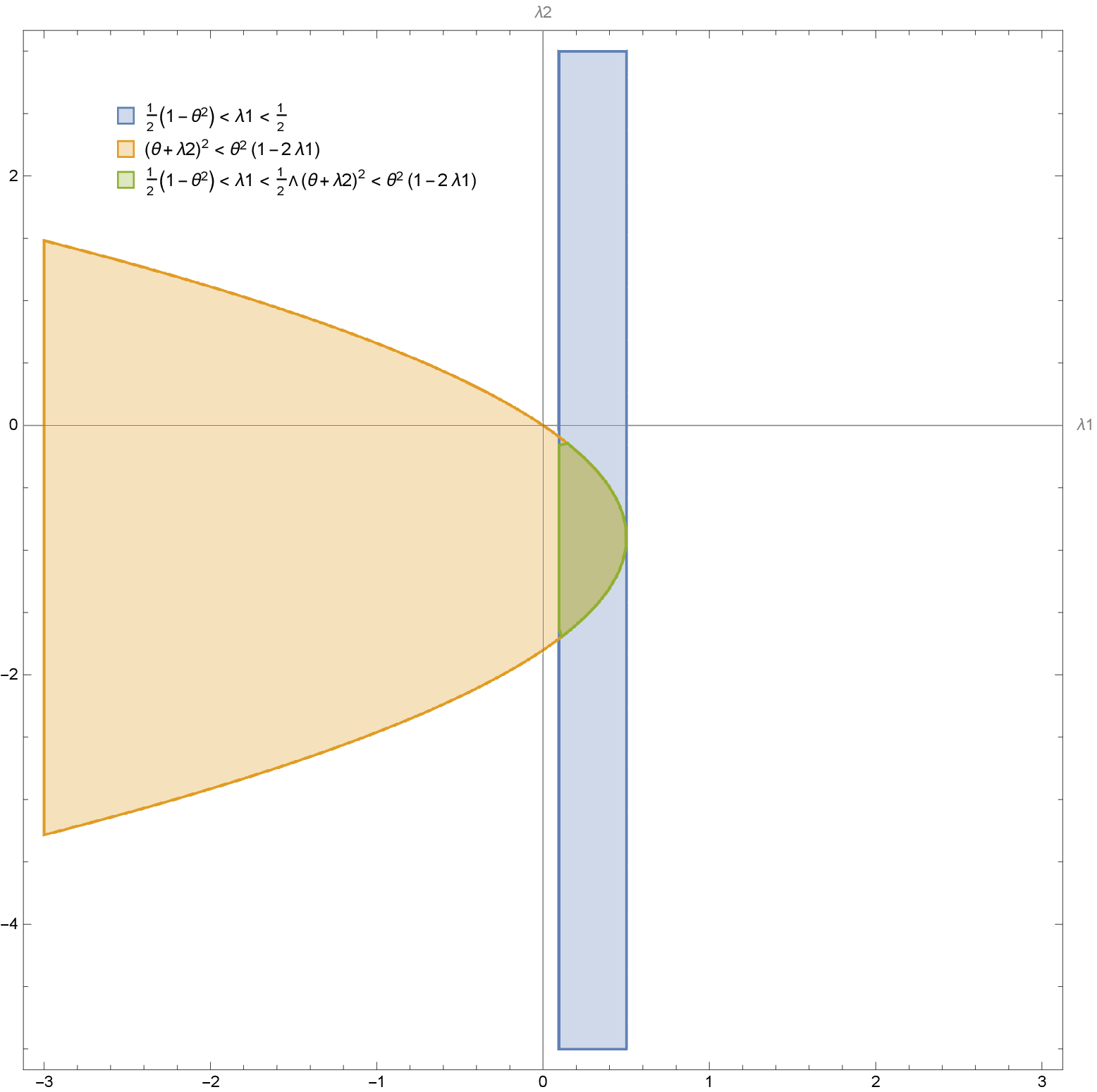}\\
   \caption{Regions $\mathcal{D}_1$ (in the left) and $\mathcal{D}_2$
(in the right) defined in \eqref{D12}
in the particular case when $\theta = 0.9$.\label{domainD1}}
   \end{center}
 \end{figure}

 \begin{figure}[ht!]
   \centering
   \includegraphics[scale = 0.5]{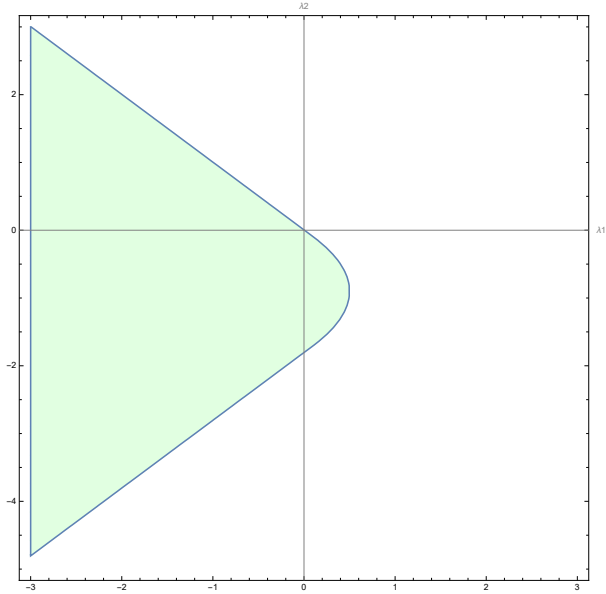}
   \caption{Region $\mathcal{D} = \mathcal{D}_1 \cup \mathcal{D}_2$, when
$\theta = 0.9$. \label{domainD3}}
 \end{figure}

  The knowledge of the domain where the matrix $D_{n, \boldsymbol{\lambda}}$ is
positive definite, allows one to give continuity to the computations of
$L_n(\cdot,\cdot)$ and its limiting function when $n \to \infty$. It is shown
in Bryc and Dembo \cite{bryc97} (see page 330), for the special case $\theta =
0$, that
  \begin{equation*}
   \lim_{n \to \infty} L_n(\lambda_1, \lambda_2) =
   \begin{cases}
    -\frac{1}{2}\log\left(\frac{1 - 2 \lambda_1 + \sqrt{(1 - 2 \lambda_1)^2 - 4
\lambda_2^2}}{2}\right),& \mbox{if } (\lambda_1,\lambda_2) \in
\mathcal{D}_\varphi,\\
    \hspace{3mm} \infty,& \mbox{otherwise},
   \end{cases}
  \end{equation*}
  where
  \begin{equation*}
   \mathcal{D}_\varphi = \{(\lambda_1, \lambda_2) \in \R^2 \,|\, \sup_{\omega
\in \T} \varphi_{\boldsymbol{\lambda}}(\omega) < 1/2\}.
  \end{equation*}
  Even though representing a particular degenerate case, it is important to
note such result. If $\theta = 0$, the process $(X_n)_{n \in \N}$ in \eqref{ar1}
reduces itself to an i.i.d.\ sequence of random variables with
standard Gaussian distribution. We shall generalize the result in
Bryc and Dembo \cite{bryc97} on a bivariate setting, for the case when $\theta
\neq 0$.

  \begin{lemma}\label{lemma2}
   Let $L_n(\cdot,\cdot)$ denote the normalized cumulant generating function of
$(\W_n)_{n \geq 2}$, then
   \begin{equation}\label{convLnL}
    \lim_{n \to \infty} L_n(\lambda_1, \lambda_2) = L(\lambda_1, \lambda_2),
   \end{equation}
   where $L:\R^2 \rightarrow \R \cup \{\infty\}$ is defined by
   \begin{equation}\label{ar1lcgf}
    L(\lambda_1, \lambda_2) =
    \begin{cases}
     \displaystyle -\frac{1}{2} \log \left(\frac{1 + \theta^2 - 2\lambda_1 +
\sqrt{(1 + \theta^2 - 2\lambda_1)^2 - 4 (\theta + \lambda_2)^2}}{2}\right),&
\mbox{ for } (\lambda_1, \lambda_2) \in \mathcal{D},\\
     \hspace{3mm} \infty, & \mbox{otherwise,}
    \end{cases}
   \end{equation}
   with the domain $\mathcal{D}$ given in Lemma \ref{domD}.
  \end{lemma}

  \begin{proof}
   Let $(\alpha_{n,k}^{\boldsymbol{\lambda}})_{k=1}^n$ represent the sequence
of
eigenvalues of $T_n(\varphi_{\boldsymbol{\lambda}})\,T_n(g_\theta)$, with
$g_\theta(\cdot)$ denoting the spectral density function, defined in
\eqref{spectral}, and $\varphi_{\boldsymbol{\lambda}}(\cdot)$ the function
given in
\eqref{philambda}. If $(\lambda_1, \lambda_2) \in \mathcal{D}$, Lemma
\ref{domD} guarantees that $\alpha_{n,k}^{\boldsymbol{\lambda}} < 1/2$, for all
$1
\leq k \leq n$ and $n$ large enough. Then, from \eqref{mom} and \eqref{Ln3} it
follows that
 \begin{equation}\label{Ln4}
   L_n(\lambda_1, \lambda_2) =  -\frac{1}{2n} \sum_{k=1}^n \log\left(1 - 2
\alpha_{n,k}^{\boldsymbol{\lambda}}\right),
\quad \mbox{for } (\lambda_1, \lambda_2) \in \mathcal{D} \mbox{ and $n$ large
enough}.
   \end{equation}

Nonetheless, if $(\lambda_1, \lambda_2) \notin \mathcal{D}$ and
$(\alpha_{n,n}^{\boldsymbol{\lambda}})_{n \in \N}$ represents the sequence of
maximum eigenvalues of $T_n(\varphi_{\boldsymbol{\lambda}})\,T_n(g_\theta)$, we
can always find a subsequence $(\alpha_{n_j,n_j}^{\boldsymbol{\lambda}})_{j \in
\N}$ of
$(\alpha_{n,n}^{\boldsymbol{\lambda}})_{n \in \N}$ such that
$\alpha_{n_j,n_j}^{\boldsymbol{\lambda}} \geq 1/2$, for all $j \in \N$. In that
case, we
have $M_{Z_{n_j,n_j}}(\alpha_{n_j,n_j}^{\boldsymbol{\lambda}}) = \infty$, for
all $j
\in \N$, implying that $\lim_{n\to\infty} L_n(\lambda_1,\lambda_2) = \infty$.
Henceforth, we only need to take care when $(\lambda_1, \lambda_2)$ belongs to
$\mathcal{D}$, because in this case, $L_n(\lambda_1,\lambda_2)$ is finite for
$n$ large enough and it is given by \eqref{Ln4}.

Consider in what follows the measure space $L^\infty(\T) := L^\infty(\T,
\mathcal{B}(\T), \nu)$, were $\nu(\cdot)$ is the Lebesgue measure acting on
$\mathcal{B}(\T)$, the  Borel $\sigma$-algebra over $\T$. If $h \in
L^\infty(\T)$, the usual norm $||h||_\infty = $
\begin{equation*}
\inf\{h^{*\boldsymbol{S}(N)}\, | \, N \in
\mathcal{B}(\mathbb{T}),\, \nu(N) = 0\}, \ \mbox{ with } \
h^{*\boldsymbol{S}(N)} = \sup\{|h(x)|\,:\, x \notin N\},
\end{equation*}
shall be considered. Since $\varphi_{\boldsymbol{\lambda}}, \,g_\theta \in
L^\infty(\T)$, it is straightforward to show that (see Avram \cite{avram88})
\begin{equation}\label{bounded}
 |\alpha_{n, k}^{\boldsymbol{\lambda}}| \leq
||\varphi_{\boldsymbol{\lambda}}||_\infty
 ||g_\theta||_\infty, \quad \mbox{for all } 1 \leq k \leq n \mbox{ and } n \in
\N.
\end{equation}

   Let $m_{\varphi_{\boldsymbol{\lambda}} \, g_\theta}$ and
$M_{\varphi_{\boldsymbol{\lambda}}
\, g_\theta}$ denote, respectively, the \emph{essential infimum} and
\emph{essential suppremum} (see Grenander and Szeg\"{o} \cite{grenander58}) of
the continuous
mapping $\varphi_{\boldsymbol{\lambda}}\,g_\theta: \T \rightarrow \R$,
belonging to
$L^\infty(\T)$ and defined by
   \begin{equation}\label{quotient}
    (\varphi_{\boldsymbol{\lambda}}\,g_\theta)(\omega) =
\varphi_{\boldsymbol{\lambda}}(\omega)\, g_\theta(\omega) =
\frac{\lambda_1+\lambda_2
\cos(\omega)}{1+\theta^2-2\theta\cos(\omega)}.
   \end{equation}
   The function $(\varphi_{\boldsymbol{\lambda}}\,g_\theta)(\cdot)$ is
continuous
and bounded in $[-\pi, \pi]$, hence it attains a maximum and a minimum in that
interval. It follows that
   \begin{equation*}
    m_{\varphi_{\boldsymbol{\lambda}} \, g_\theta} = \min_{\omega \in
\T}\{(\varphi_{\boldsymbol{\lambda}} \, g_\theta)(\omega)\} \quad  \mbox{and}
\quad
M_{\varphi_{\boldsymbol{\lambda}} \, g_\theta}=\max_{\omega \in
\T}\{(\varphi_{\boldsymbol{\lambda}} \, g_\theta)(\omega)\}.
   \end{equation*}
   Since
   \begin{equation*}
    \frac{d}{d \omega} (\varphi_{\boldsymbol{\lambda}}\,g_\theta)(\omega) =
-\frac{\lambda_2 \sin(\omega)}{1+\theta^2-2\theta \cos(\omega)} -
\frac{2\theta(\lambda_1+\lambda_2
\cos(\omega))\sin(\omega)}{(1+\theta^2-2\theta\cos(\omega))^2},
   \end{equation*}
   we notice that $(\varphi_{\boldsymbol{\lambda}}\,g_\theta)(\omega)$ has two
critical
points at $\omega=-\pi$ and $\omega=0$. Moreover,
   \begin{itemize}
    \item if $\lambda_2 < -2\theta \lambda_1/(1+\theta^2)$, then $\frac{d^2}{d
\omega^2} (\varphi_{\boldsymbol{\lambda}}\,g_\theta)(-\pi) < 0$ and
$\frac{d^2}{d
\omega^2} (\varphi_{\boldsymbol{\lambda}}\,g_\theta)(0) > 0$;
    \item if $\lambda_2 > -2\theta \lambda_1/(1+\theta^2)$, then $\frac{d^2}{d
\omega^2} (\varphi_{\boldsymbol{\lambda}}\,g_\theta)(-\pi) > 0$ and
$\frac{d^2}{d
\omega^2} (\varphi_{\boldsymbol{\lambda}}\,g_\theta)(0) < 0$;
    \item if $\lambda_2 = -2\theta \lambda_1/(1+\theta^2)$, then
$(\varphi_{\boldsymbol{\lambda}}\, g_\theta)(\omega) = \lambda_1$ is constant.
   \end{itemize}
   Therefore, since
   \begin{equation*}
    (\varphi_{\boldsymbol{\lambda}}\, g_\theta)(-\pi) =
\frac{\lambda_1-\lambda_2}{1+\theta^2+2\theta} \quad \mbox{and} \quad
(\varphi_{\boldsymbol{\lambda}}\, g_\theta)(0) =
\frac{\lambda_1+\lambda_2}{1+\theta^2-2\theta},
   \end{equation*}
   we conclude that
   \begin{equation*}
    m_{\varphi_{\boldsymbol{\lambda}} \, g_\theta} =
    \begin{cases}
     \frac{\lambda_1 + \lambda_2}{1+ \theta^2 - 2 \theta},& \mbox{if } \lambda_2
< -\frac{2 \theta \lambda_1}{1+\theta^2},\\
     \frac{\lambda_1 - \lambda_2}{1+ \theta^2 + 2 \theta},& \mbox{if } \lambda_2
\geq -\frac{2 \theta \lambda_1}{1+\theta^2},
    \end{cases}
     \quad \mbox{and} \quad
    M_{\varphi_{\boldsymbol{\lambda}} \, g_\theta} =
    \begin{cases}
     \frac{\lambda_1 - \lambda_2}{1+ \theta^2 + 2 \theta},& \mbox{if } \lambda_2
< -\frac{2 \theta \lambda_1}{1+\theta^2},\\
     \frac{\lambda_1 + \lambda_2}{1+ \theta^2 - 2 \theta},& \mbox{if } \lambda_2
\geq -\frac{2 \theta \lambda_1}{1+\theta^2}.
    \end{cases}
   \end{equation*}

   Considering the case in which $(\lambda_1, \lambda_2) \in \mathcal{D}$, it
follows that
   \begin{equation*}
    2 |\theta + \lambda_2| < 1 + \theta^2 - 2 \lambda_1,
   \end{equation*}
   whence
   \begin{equation}\label{intthe}
    -\left(\frac{1+\theta^2-2\lambda_1}{2}\right) < \theta+\lambda_2 <
\frac{1+\theta^2-2\lambda_1}{2}.
   \end{equation}
   From the left-hand side of \eqref{intthe}, we get
   \begin{equation*}
    \frac{\lambda_1-\lambda_2}{1+\theta^2+2\theta} < \frac{1}{2},
   \end{equation*}
   while from the right-hand side of \eqref{intthe}, we obtain
   \begin{equation*}
    \frac{\lambda_1+\lambda_2}{1+\theta^2-2\theta} < \frac{1}{2}.
   \end{equation*}
   Hence, we conclude that $M_{\varphi_{\boldsymbol{\lambda}} \, g_\theta} <
1/2$.
On
the other hand, from
   \begin{equation*}
    ||\varphi_{\boldsymbol{\lambda}}||_\infty ||g_\theta||_\infty \geq
||\varphi_{\boldsymbol{\lambda}} \, g_\theta||_\infty =
\max\{|M_{\varphi_{\boldsymbol{\lambda}}
\, g_\theta}|, |m_{\varphi_{\boldsymbol{\lambda}} \, g_\theta}|\} \geq
-m_{\varphi_{\boldsymbol{\lambda}} \, g_\theta},
   \end{equation*}
   we conclude that $m_{\varphi_{\boldsymbol{\lambda}} \, g_\theta} \geq
-||\varphi_{\boldsymbol{\lambda}}||_\infty ||g_\theta||_\infty$. Therefore, we
just
proved that
   \begin{equation}\label{asser}
    [m_{\varphi_{\boldsymbol{\lambda}} \, g_\theta},
M_{\varphi_{\boldsymbol{\lambda}}
\,
g_\theta}] \subseteq [- ||\varphi_{\boldsymbol{\lambda}}||_\infty
||g_\theta||_\infty,
1/2).
   \end{equation}

   The denominator in the left-hand side of \eqref{quotient} satisfies
   \begin{equation*}
    \inf_{\omega \in \T}|1 + \theta^2 - 2 \theta \cos(\omega)| =
\min\{1+\theta^2-2\theta,1+\theta^2+2\theta\} > 0, \quad \mbox{for all }\,
\theta \in (-1,1).
   \end{equation*}
   Then, it follows from theorem 5.1 in Tyrtyshnikov \cite{tyrtyshnikov94}
that, if $F$ is
any arbitrary continuous function with bounded support (i.e., the set of those
$x \in \R$ for which $F(x) \neq 0$ is bounded), we get
   \begin{equation}\label{convTyrty}
    \lim_{n \to \infty} \frac{1}{n} \sum_{k=1}^n
F(\alpha_{n,k}^{\boldsymbol{\lambda}})
= \frac{1}{2\pi} \int_{\T}
(F\circ(\varphi_{\boldsymbol{\lambda}}\,g_\theta))(\omega)\,d\omega.
   \end{equation}
   In particular, the latter convergence applies itself when considering the
continuous function \\$F:[- ||\varphi_{\boldsymbol{\lambda}}||_\infty
||g_\theta||_\infty, 1/2) \rightarrow \R$ defined by
   \begin{equation*}
    F(x) = -\frac{\log(1 - 2 x)}{2}.
   \end{equation*}
   Indeed, from \eqref{bounded} and \eqref{asser}, combined with the
result of Lemma \ref{domD}, we conclude that $F(\cdot)$ has bounded support and
that $F(\alpha_{n,k}^{\boldsymbol{\lambda}})$ are finite, for every $1 \leq k
\leq
n$
and $n$ large enough. Besides that, $(F\circ(\varphi_{\boldsymbol{\lambda}}\,
g_\theta))(\omega) = \log[1 -2\,(\varphi_{\boldsymbol{\lambda}}\,
g_\theta)(\omega)]$
is finite, for every $\omega \in \T$, due to \eqref{asser}.
Therefore, the two sides of \eqref{convTyrty} are well defined and such
convergence holds, giving
   \begin{equation*}
    \begin{split}
    \lim_{n \to \infty} L_n(\lambda_1, \lambda_2) &= \lim_{n \to \infty}
-\frac{1}{2n} \sum_{k=1}^n \log(1 - 2 \alpha_{n,k}) = \lim_{n \to \infty}
\frac{1}{n} \sum_{k=1}^n F(\alpha_{n,k})\\
    & = -\frac{1}{2\pi} \int_{\T} (F\circ(\varphi_{\boldsymbol{\lambda}}\,
g_\theta))(\omega) \, d \omega = -\frac{1}{4\pi} \int_\T \log(1 -
2\,(\varphi_{\boldsymbol{\lambda}} \, g_\theta)(\omega))\,d\omega \\
    & = -\frac{1}{2} \log \left(\frac{1 + \theta^2 - 2\lambda_1 + \sqrt{(1 +
\theta^2 - 2\lambda_1)^2 - 4 (\theta + \lambda_2)^2}}{2}\right),
   \end{split}
   \end{equation*}
 where the last equality was achieved using equation 4.224(9) in
Gradshteyn and Ryzhik \cite{gradshteyn07}.
  \end{proof}

  \subsection{LDP of the random sequence $\boldsymbol{(\W_n)_{n \geq 2}}$}

  Here we use the lemmas of Subsection \ref{analysisNCG}, combined with the
G\"{a}rtner-Ellis' theorem, to prove that the sequence $(\W_n)_{n \geq 2}$ in
\eqref{Wn} satisfies an LDP. There are two conditions that must be
satisfied in order to apply the G\"{a}rtner-Ellis' theorem (see pages 43-44 in
Dembo and Zeitouni \cite{dembo10}).
  \begin{itemize}
   \item \textbf{Condition A:} for each $(\lambda_1, \lambda_2) \in \R^2$, the
limiting cumulant generating function $L(\cdot,\cdot)$, defined as the limit in
\eqref{convLnL} and explicitly given by \eqref{ar1lcgf}, exists as an
extended real number. Moreover, if
   \begin{equation}\label{domL}
    \mathcal{D}_L = \big\{(\lambda_1, \lambda_2) \in \R^2 \, | \, L(\lambda_1,
\lambda_2) < \infty\big\}
   \end{equation}
   denotes the \emph{effective domain} of $L(\cdot, \cdot)$, the origin must
belong to $\mathcal{D}_L^\circ$ (the interior of $\mathcal{D}_L$).

   \item \textbf{Condition B:} $L(\cdot,\cdot)$ is an \emph{essentially smooth}
function, that is,
   \begin{enumerate}
    \item $\mathcal{D}_L^\circ$ is non-empty;
    \item $L(\cdot,\cdot)$ is differentiable throughout $\mathcal{D}_L^\circ$;
    \item $L(\cdot, \cdot)$ is \emph{steep}, i.e., we get $\lim_{n \to
\infty}||\nabla
L(\lambda_{1,n},\lambda_{2,n})|| = \infty$, in the case
$(\lambda_{1,n},\lambda_{2,n})_{n \in \N}$ is a sequence in
$\mathcal{D}_L^\circ$ converging to a boundary point of $\mathcal{D}_L^\circ$,
where $||(x,y)|| = \sqrt{x^2 + y^2}$ denotes the usual Euclidean norm in
$\R^2$.
   \end{enumerate}
  \end{itemize}

  Note that, if \textbf{Condition A} above is satisfied, then
\textbf{Condition B.1} is redundant. In the following proposition, we verify
that both \textbf{Conditions A} and \textbf{B} are satisfied when considering
the LDP for $(\W_n)_{n \geq 2}$.  The cornerstone of our proof stands on the
observation that the effective domain $\mathcal{D}_L$, defined in \eqref{domL},
contains the domain $\mathcal{D}$, given in Lemma \ref{domD}.

\medskip

  \begin{prop} \label{SupConj}
   The sequence of random vectors $(\W_n)_{n \geq 2}$, defined in \eqref{Wn},
satisfies an LDP with good rate function
   \begin{equation}\label{ratear1}
    J(x,y) = \begin{cases}
              \frac{1}{2}\left[x(1 + \theta^2) - 1 - 2 y \theta +
\log\left(\frac{x}{x^2-y^2}\right)\right],& \mbox{for } 0 < x \mbox{ and  } |y|
< x,\\
              \infty,& \mbox{otherwise}.
             \end{cases}
   \end{equation}
  \end{prop}

  \begin{proof}
   Let $L(\cdot,\cdot)$ denote the function in \eqref{ar1lcgf} and
$\mathcal{D}$ the domain defined in Lemma \ref{domD}. The effective
domain of $L(\cdot, \cdot)$ is given by
   \begin{equation*}
    \mathcal{D}_L = \left\{(\lambda_1, \lambda_2) \in \R^2 \, | \, \lambda_1 <
\frac{1}{2} \, , \, 4 (\theta + \lambda_2)^2 < (1 + \theta^2 - 2 \lambda_1)^2
\right\}.
   \end{equation*}
   Notice that $\mathcal{D}$ is a proper subset of $\mathcal{D}_L$.
Furthermore, if $(\lambda_1,\lambda_2)=(0,0)$, then
   \begin{equation*}
    0 < (1-\theta^2)^2 \Rightarrow 0 < 1-2\theta^2+\theta^4 \Rightarrow
4\theta^2 < (1 + \theta^2)^2.
   \end{equation*}
   Whence, the origin $(0,0) \in \R^2$ belongs to the interior
of $\mathcal{D}_L$, for any $\theta \in (-1,1)$, proving that \textbf{Condition
A} above is fulfilled. The proof that
$L(\cdot,\cdot)$ is
an essentially smooth function follows the same steps as the proof given in
section 3.6 of Bryc and Dembo \cite{bryc97}, so that \textbf{Condition B} is
also verified.

   Let $J: \R^2 \rightarrow \R$ denote the Fenchel-Legendre dual of $L(\cdot,
\cdot)$, defined by the suppremum
$$ J(x, y)  = \sup_{(\lambda_1, \lambda_2) \in \R^2} \big\{x \lambda_1 + y
\lambda_2 - L(\lambda_1, \lambda_2)\big\}=$$
\begin{equation}\label{FLL}
\sup_{(\lambda_1, \lambda_2) \in \mathcal{D}} \left\{x \lambda_1 + y
\lambda_2 + \frac{1}{2} \log \left(\frac{1 + \theta^2 - 2\lambda_1 + \sqrt{(1 +
\theta^2 - 2\lambda_1)^2 - 4 (\theta + \lambda_2)^2}}{2}\right)\right\}.
   \end{equation}
   From the G\"{a}rtner-Ellis' theorem, $(\W_n)_{n \geq 2}$ satisfies
an LDP with good rate function $J(\cdot,\cdot)$. To explicitly compute
$J(\cdot,\cdot)$, consider the auxiliary function $K:
\mathcal{D} \rightarrow \R$,
defined by
   \begin{equation*}
    K(\lambda_1, \lambda_2)= x \lambda_1 + y \lambda_2 - L(\lambda_1,
\lambda_2), \quad \mbox{with } (x,y) \in \R^2.
   \end{equation*}
   The partial derivatives of $K(\cdot, \cdot)$ are
   \begin{equation*}
    K_{\lambda_1}(\lambda_1, \lambda_2) = x-\frac{1}{\sqrt{1 - 4 \lambda_1 + 4
\lambda_1^2 - 4 \lambda_2^2 - 8 \theta \lambda_2 - 2 \theta^2 (2 \lambda_1 + 1)
+ \theta^4}}
   \end{equation*}
   and
   \begin{equation*}
   K_{\lambda_2}(\lambda_1,\lambda_2) = y +
 \frac{-2 (\theta +
\lambda_2)}{\left(1 + \theta^2 - 2 \lambda_1 + \sqrt{\left(1+\theta ^2-2
\lambda_1\right)^2 - 4 (\theta + \lambda_2)^2}\right)\sqrt{\left(1+\theta^2 -
2\lambda_1\right)^2-4 (\theta +\lambda_2)^2}}.
   \end{equation*}
   Provided that $x > 0$ and $x>|y|$, the solution to the system of equations
   \begin{equation*}
    \begin{cases}
     K_{\lambda_1}(\lambda_1,\lambda_2) = 0,\\
     K_{\lambda_2}(\lambda_1,\lambda_2) = 0,
    \end{cases}
   \end{equation*}
   is given by
   \begin{equation*}
    \lambda_1^* = \frac{1+\theta^2}{2} - \frac{x^2 + y^2}{2x(x^2 - y^2)} \quad
\mbox{and} \quad \lambda_2^* = \frac{y}{x^2 - y^2} - \theta.
   \end{equation*}
   It is not difficult to prove that $(\lambda_1^*, \lambda_2^*)$ is the point
where the suppremum in \eqref{FLL} is attained. Hence, it follows that
   \begin{equation*}
    J(x,y) = K(\lambda_1^*,\lambda_2^*) = \frac{1}{2}\left[x(1+\theta^2) - 1 - 2
 y \theta + \log\left(\frac{x}{x^2-y^2}\right)\right],
   \end{equation*}
   where $0 < x \mbox{ and
} |y| < x.$

   Note that, the restrictions $0 < x$ and $|y| < x$ are related to the
inequalities $0 < \sum_{k=1}^n X_k^2$ (see McLeod and Jim\'{e}nez
\cite{mcleod84}) and
$|\sum_{k=2}^n X_k X_{k-1}| < \sum_{k=1}^n X_k^2 $.

   If $x \leq 0$ or $|y|  \geq x$, we may define $J(x,y) = \infty$, since
$K(\lambda_1, \lambda_2)$ is unbounded. Indeed, if $x < 0$, then
   \begin{equation*}
    \lim_{\lambda_1 \to -\infty} K(\lambda_1, \lambda_2) \approx \lim_{\lambda_1
\to -\infty} x \lambda_1 = \infty,
   \end{equation*}
   because the linear part $x \lambda_1$ rules over the logarithmic part of
$K(\lambda_1, \lambda_2)$, while if $x =0$, then
  \begin{equation*}
   \lim_{\lambda_1 \to -\infty} K(\lambda_1, \lambda_2) = \lim_{\lambda_1
\to -\infty} y \lambda_2 - L(\lambda_1,\lambda_2)= \infty.
  \end{equation*}
If $x > 0$, but $|y| \geq x$, then we have two cases to consider: the first one
is when $y \leq -x$, whereby
   \begin{equation*}
    \lim_{\lambda_2 \to -\infty} K(\lambda_1, \lambda_2) \approx \lim_{\lambda_2
\to -\infty} y \lambda_2 = \infty;
   \end{equation*}
   the second case is when $y \geq x$, for which it follows that
   \begin{equation*}
    \lim_{\lambda_2 \to \infty} K(\lambda_1, \lambda_2) \approx \lim_{\lambda_2
\to \infty} y \lambda_2 = \infty.
   \end{equation*}
  \end{proof}

  A graph of the function $J(\cdot,\cdot)$, in \eqref{ratear1}, is shown
in Figure \ref{graphRate}, when $\theta = 0.3$. Since $L(\cdot, \cdot)$
is a convex function, $J(\cdot,\cdot)$ must also be a convex function (see
section VI.5 in Ellis \cite{ellis85}).

  \begin{figure}[ht!]
   \begin{center}
   \includegraphics[scale = 0.38]{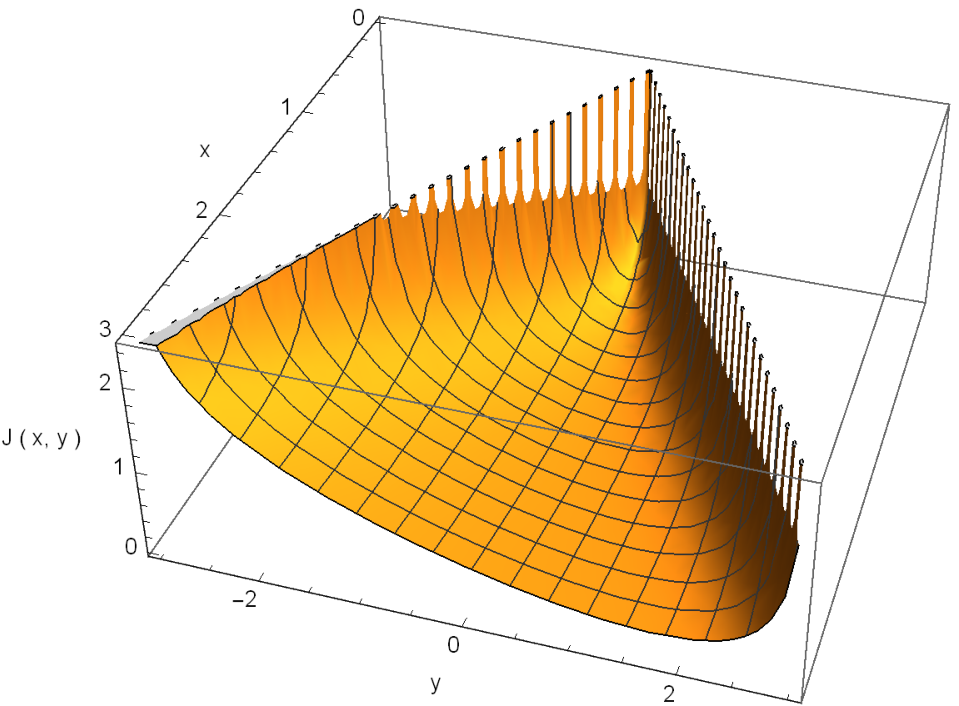} \quad
   \includegraphics[scale = 0.38]{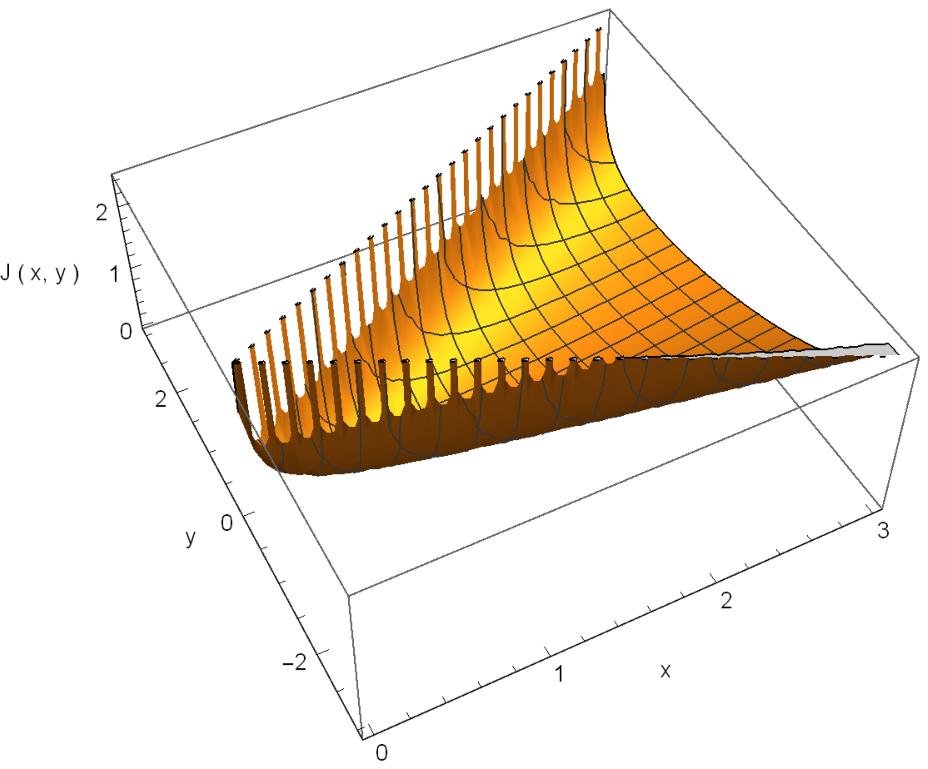}
   \caption{Graph of the function $J(x,y)$ given in \eqref{ratear1} when
$\theta = 0.3$, $x \in (0,3]$ and $y \in (-3, 3)$. \label{graphRate}}
   \end{center}
  \end{figure}

\section{Particular cases}
  We dedicate this section to show three particular examples where the
reasoning of the last section can be used, via Contraction Principle, to get
explicit rate functions for univariate random sequences. Two of these examples
were already known from Bercu {\it et al}.\ \cite{bercu97} and Bryc and
Smolenski \cite{bryc93}. We shall obtain them
as a continuous transform of the random vector $\W_n$, defined in \eqref{Wn}.
In Subsection \ref{ldpFOEA}, we present a result which we believe
is new in the literature.

The Contraction Principle will be of great importance for the computations of
the rate functions.

\begin{thm}[Contraction Principle]
 If a sequence of random vectors $(\boldsymbol{V}_n)_{n \in \N}$ with values in
$E
\subseteq \R^d$ satisfies an LDP with good rate function
$J(\boldsymbol{\cdot}):\R^d
\rightarrow [0, \infty]$ and $U_n = f(\boldsymbol{V}_n)$, where
$f(\boldsymbol{\cdot}):
E
\rightarrow
\R$ is a continuous function, then the random sequence $(U_n)_{n \in \N}$ also
satisfies an LDP with good rate function $I(\cdot):\R \rightarrow [0, \infty]$
given by
   \begin{equation*}
    I(c) = \inf_{\boldsymbol{x} \in E} \big\{J(\boldsymbol{x})\,|\, \mbox{with }
\boldsymbol{x} \mbox{
such that } f(\boldsymbol{x}) = c \big\}, \quad \mbox{for all } c \in \R.
   \end{equation*}
  \end{thm}
  \begin{proof}
   See section 4.2.1 in Dembo and Zeitouni \cite{dembo10}.
  \end{proof}

  Since the sequence of random vectors $(\W_n)_{n \geq 2}$ satisfies an LDP
with rate function $J(\cdot,\cdot)$, given in \eqref{ratear1}, the
Contraction Principle ensures that any sequence of vectors $(f(\W_n))_{n \geq
2}$, for $f: \R^2 \rightarrow \R$ continuous, satisfies an LDP with good rate
function
  \begin{equation}\label{Rate2}
   I(c) = \inf_{(x,y) \in \R^2} \big\{J(x,y)\,|\, \mbox{with } (x,y) \mbox{
such that } f(x,y) = c \big\}, \quad \mbox{for all } c \in \R.
  \end{equation}

There is a standard procedure involving Calculus techniques for computing the
infimum in \eqref{Rate2}, namely, checking for the critical points of the
derivatives from $J(\cdot,\cdot)$. In the examples considered below, the
Wolfram Mathematica software (version 11.2.0.0) was used in the calculations.

Note that, $f(\W_n) = f\left(\frac{1}{n}\left(\sum_{k=1}^n X_k^2,
\sum_{k=2}^n X_k X_{k-1}\right)\right)$ is a continuous function involving only
the components $\frac{1}{n} \sum_{k=1}^n X_k^2$ and $\frac{1}{n} \sum_{k=2} X_k
X_{k-1}$. Any statistic that can be written in terms of these components, as a
continuous transform of $\W_n$, is suitable for our method. In particular, in
Sections 3.1-3.3 we shall consider $f: \R^2 \rightarrow \R$ as
being respectively defined by
  \begin{enumerate}
   \item $f(x,y) = x$;
   \item $f(x,y) = y$;
   \item $f(x,y) = \frac{y}{x}$, for $x > 0$.
  \end{enumerate}
  Other continuous functions $f: \R^2 \rightarrow \R$ could also be considered,
however, in the present work, we shall restrict our attention to these three
cases.

\subsection{LDP for the quadratic mean}\label{ldpQM}
 Consider the Quadratic Mean of a random sample $X_1, \cdots, X_n$ which
satisfies \eqref{ar1}, given by
\begin{equation*}
    \tilde\gamma_n(0) = \frac{1}{n} \sum_{k=1}^n X_k^2.
\end{equation*}

 Bryc and Smolenski \cite{bryc93} proved that the sequence
$(\tilde\gamma_n(0))_{n \in \N}$ satisfies an LDP with rate function given by
   \begin{equation}\label{IBryc}
    \mathbb{I}(c) =
     \begin{cases}
       \frac{1}{2}\left[ c \left(1 + \theta^2\right) - \sqrt{1 + 4 \theta^2 c^2}
- \log \left(\frac{2c}{1 + \sqrt{1 + 4 \theta^2 c^2}}\right) \right], & \mbox{if
} c > 0,\\
      \infty, & \mbox{if } c \leq 0.
     \end{cases}
   \end{equation}
   Here we obtain the result from Bryc and Smolenski \cite{bryc93} as a
particular case, by
using Proposition \ref{SupConj} and the Contraction Principle.

   Note that $\tilde\gamma_n(0)$ may be obtained as the projection on the
first coordinate of the vector $\W_n$, given in \eqref{Wn}. Consider
$f_1: \R^2 \rightarrow \R$ the continuous function given by $f_1(x,y) = x$. Then
$f_1(\W_n) = \tilde\gamma_n(0)$ and, since $(\W_n)_{n \geq 2}$
satisfies an LDP with rate function $J(\cdot, \cdot)$, given in \eqref{ratear1},
the Contraction Principle ensures that $(\tilde\gamma_n(0))_{n
\in \N}$ satisfies an LDP with rate function, which we shall denote by $I_1: \R
\rightarrow [0, \infty]$. Then $I_1(\cdot)$ can be computed from \eqref{Rate2}
in
the following way.
   \begin{itemize}
   \item By the Contraction Principle, if $c > 0$, then
   \begin{equation}\label{I1c}
    \begin{split}
    I_1(c) & = \inf_{\{0 < x , \ |y| < x\}} \{J(x,y) \,|\,  f_1(x,y) = c\} =
\inf_{|y| < c} \{J(c,y)\}\\
    & = \inf_{|y| < c}\left\{\frac{1}{2}\left[c (1 + \theta^2) - 1 - 2 y
\theta + \log\left(\frac{c}{c^2 - y^2}\right) \right]\right\};
    \end{split}
   \end{equation}
   \item The infimum in \eqref{I1c} is attained at
  \begin{equation*}
   y_c =
   \begin{cases}
    \displaystyle \frac{-1 + \sqrt{1 + 4 c^2 \theta^2}}{2 \theta}, & \mbox{if }
0 < |\theta| < 1,\\
    0, & \mbox{if } \theta = 0;
   \end{cases}
  \end{equation*}

  \item  If $\theta = 0$, then it immediately follows that
  \begin{equation*}
   I_1(c) = \frac{c - 1 - \log (c)}{2};
  \end{equation*}

  \item  If $0 < |\theta| < 1$, then, after some algebraic computations, we
obtain
  \begin{equation*}
   \begin{split}
   I_1(c) &= \frac{1}{2}\left[c (1 + \theta^2) - 1 - 2 y_c \theta +
\log\left(\frac{c}{c^2 - y_c^2}\right)\right]\\
   &= \frac{1}{2} \left[ c(1 + \theta^2) - \sqrt{1 + 4 c^2 \theta^2} + \log
\left( \frac{2 c \theta^2}{\sqrt{1 + 4 c^2 \theta^2} - 1} \right)\right],
   \end{split}
  \end{equation*}
  and since
  \begin{equation*}
   \begin{split}
   &\log\left(\frac{2c\theta^2}{\sqrt{1 + 4 c^2 \theta^2} - 1}\right) =
\log\left(\frac{2c\theta^2(\sqrt{1 + 4 c^2 \theta^2} + 1)}{(\sqrt{1 + 4 c^2
\theta^2} - 1)(\sqrt{1 + 4 c^2 \theta^2} + 1)}\right) \\
   & = \log\left(\frac{2c\theta^2(\sqrt{1 + 4 c^2 \theta^2} + 1)}{1 + 4 c^2
\theta^2 - 1}\right) = \log\left(\frac{\sqrt{1 + 4 c^2 \theta^2} + 1}{2
c}\right)  = - \log\left(\frac{2 c}{\sqrt{1 + 4 c^2 \theta^2} + 1}\right),
   \end{split}
  \end{equation*}
  we obtain
  \begin{equation*}
   I_1(c) = \frac{1}{2} \left[ c(1 + \theta^2) - \sqrt{1 + 4 c^2 \theta^2} -
\log \left(\frac{2 c}{\sqrt{1 + 4 c^2 \theta^2} + 1} \right)\right];
  \end{equation*}

  \item Considering that $I_1(c) = \infty$, for $c \leq 0$, we conclude that
$I_1(c) = \mathbb{I}(c)$, for all $c \in \R$, with $\mathbb{I}(\cdot)$ defined
in \eqref{IBryc}.
  \end{itemize}

  Therefore, we get the same result as in expression (1.2) in Bryc and
Smolenski \cite{bryc93}.
The graphs of $I_1(\cdot)$ are illustrated in Figure \ref{I1} for four
different values of $\theta$. Notice that $I_1(\cdot)$ is symmetric with
respect to the values of $\theta$, i.e., $I_1(\cdot)$ is the same function for
$\theta$ and $-\theta$, given that $\theta \in (0,1)$.

  \begin{figure}[ht!]
   \begin{center}
   \includegraphics[scale = 0.4]{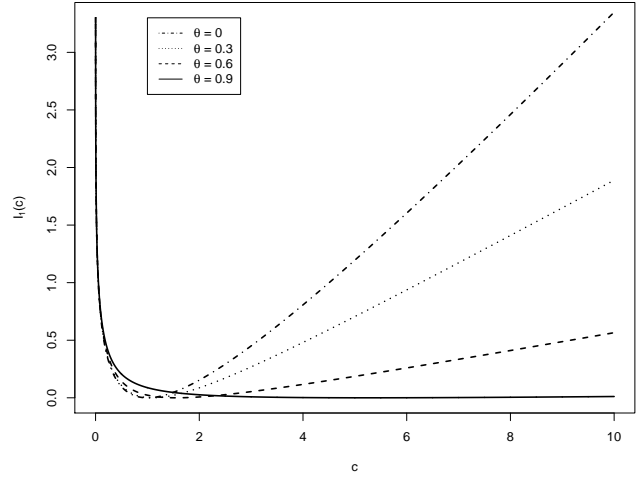}\caption{Graphs of
$I_1(\cdot)$ for $\theta \in \{0, \, 0.3, \, 0.6, \, 0.9\}$ and $c \in (0,
10]$. \label{I1}}
   \end{center}
  \end{figure}

\subsection{LDP for the first order empirical autocovariance}\label{ldpFOEA}

Consider now the first order Empirical Autocovariance of $X_1,
\cdots, X_n$, defined below as
\begin{equation*}
  \tilde\gamma_n(1) = \frac{1}{n} \sum_{k=2}^n X_k X_{k+1}.
\end{equation*}
By the same reasoning as for the Quadratic Mean, we show here that
the sequence $(\tilde\gamma_n(1))_{n \geq 2}$ satisfies an LDP, under the
assumption that $(X_n)_{n \in \N}$ follows an AR(1) process, as defined in
\eqref{ar1}. We present the explicit expression for the deviation function, a
result which we believe has not yet been shown in the literature.

Consider the continuous function $f_2: \R^2 \rightarrow \R$, with law
$f_2(x, y) = y$. Since $f_2(\W_n) = \tilde\gamma_n(1)$, it follows from
Proposition \ref{SupConj} and the Contraction Principle that
$(\tilde\gamma_n(1))_{n \geq 2}$ satisfies an LDP with rate function
$I_2:\R \rightarrow [0,\infty]$. To give an explicit expression for
$I_2(\cdot)$,
we proceed as follows.

\begin{itemize}
 \item By the Contraction Principle,
   \begin{equation}\label{I2c}
    \begin{split}
    I_2(c) & = \inf_{\{0 < x , \ |y| < x\}} \{J(x,y) \,|\, f_2(x,y) = c\} =
\inf_{\{0 < x, \ |c| < x\}} \{J(x,c)\}\\
    & = \inf_{\{0 < x, \ |c| < x\}}\left\{\frac{1}{2}\left[x(1 + \theta^2) -1 -2
c \theta + \log \left(\frac{x}{x^2-c^2}\right) \right]\right\};
    \end{split}
   \end{equation}
   \item
   Denote by $A(c, \theta) =$
   \begin{equation*}
     \sqrt[3]{1 + 18 c^2 \left(1+\theta ^2\right)^2+3 \sqrt{3}
\sqrt{-c^2 \left(1+\theta ^2\right)^2 \left(c^4 \left(1+\theta ^2\right)^4-11
c^2 \left(1+\theta ^2\right)^2-1\right)}}.
   \end{equation*}
   Then, the infimum in \eqref{I2c} is attained at
   \begin{equation}\label{xc}
    x_c = \frac{1 +3 c^2\left(1 + \theta^2\right)^2 +A(c,\theta)
+A\left(c,\theta\right)^2}{3 \left(1 + \theta ^2\right) A(c, \theta)},
   \end{equation}
   provided that
   \begin{equation}\label{root}
    c^2 < \frac{11 + 5 \sqrt{5}}{2(1+\theta^2)^2}.
   \end{equation}
   The condition in \eqref{root} guarantees that $A(c,
\theta)$ is real valued when $c \in \left[-\sqrt{\frac{11 + 5
\sqrt{5}}{2(1+\theta^2)^2}}, \sqrt{\frac{11 + 5
\sqrt{5}}{2(1+\theta^2)^2}}\right]$.

   \item  Inserting \eqref{xc} into \eqref{I2c}, we obtain
   \begingroup\makeatletter\def\f@size{8}\check@mathfonts
\def\maketag@@@#1{\hbox{\m@th\normalsize\normalfont#1}}
   \begin{equation*}
    \begin{split}
    &I_2(c) = \frac{1}{2}\left[x_c (1 - 2 c \theta +\theta^2) - 1 +
\log\left(\frac{1}{x_c(1 - c^2)}\right)\right] = \\
    &= \frac{1+3 c^2 (1 + \theta^2)^2 + \left(-2-6 c \theta +3 \log
\left[\frac{1 + 3 c^2 (1 + \theta ^2)^2 + A(c, \theta) +  A(c, \theta)^2}{3 (1 +
\theta ^2) A(c, \theta) \left(\frac{\left(1 + 3 c^2 \left(1 + \theta^2\right)^2
+ A(c, \theta) + A(c, \theta)^2\right)^2}{9 \left(1+\theta^2\right)^2 A(c,
\theta)^2}-c^2\right)}\right] \right) A(c, \theta) + A(c, \theta)^2}{6 A(c,
\theta)},
    \end{split}
   \end{equation*} \endgroup
   for $c^2 < \frac{11 + 5 \sqrt{5}}{2(1+\theta^2)^2}$;

   \item By setting $I_2(c) = + \infty$, if $c^2 \geq \frac{11 + 5
\sqrt{5}}{2(1+\theta^2)^2}$, we concluded that $(\tilde\gamma_n(1))_{n \geq 2}$
satisfies an LDP with rate function $I_2(\cdot)$.
\end{itemize}

 The graph of $I_2(\cdot)$ is illustrated in Figure \ref{I2} for five different
values of $\theta$.

\begin{figure}[ht!]
 \begin{center}
   \includegraphics[scale = 0.4]{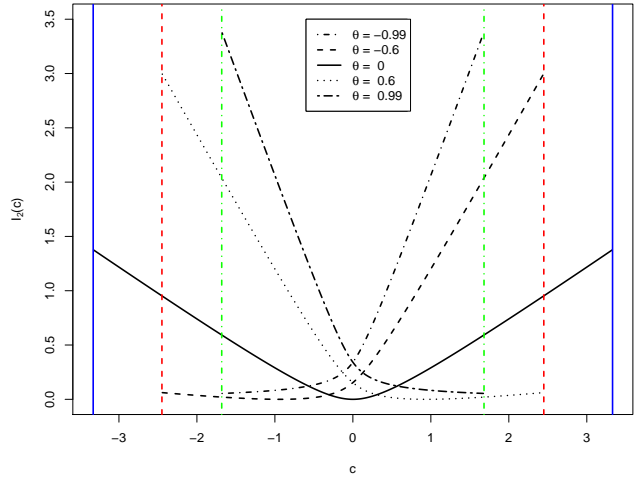}\caption{Graphs of
$I_2(\cdot)$ for $\theta \in \{-0.99, -0.6, \ 0, \, 0.6, \, 0.99\}$. The
vertical lines $c^2 = (11+5 \sqrt{5})/2$ (in blue), $c^2 = (11 + 5
\sqrt{5})/3.699$ (in red) and $c^2 = (11 + 5 \sqrt{5})/7.841$ (in green)
represent the values of $c$ where $I_2(\cdot)$ changes to $\infty$, when
$\theta = 0$, $|\theta| = 0.6$ and $|\theta| = 0.99$, respectively. \label{I2}}
 \end{center}
\end{figure}

  \subsection{LDP for the Yule-Walker estimates}

  Consider the Yule-Walker estimator
  \begin{equation}\label{thetan}
   \tilde{\theta}_n = \frac{\sum_{k=2}^n X_k X_{k-1}}{\sum_{k=1}^n X_k^2}
  \end{equation}
  of the parameter $\theta$, for
the AR(1) processes given in \eqref{ar1}. The asymptotical behavior of
such estimator is well known (see Brockwell and Davis \cite{brockwell91}), so
that
  \begin{equation*}
   \sqrt{n}(\tilde{\theta}_n - \theta) \Rightarrow \mathcal{N}(0, 1 - \theta^2)
  \end{equation*}
  and that (see Mann and Wald \cite{mann43})
  \begin{equation*}
   \tilde{\theta}_n \xrightarrow{n \to \infty} \theta, \quad \mbox{almost
surely.}
  \end{equation*}

  In Bercu {\it et al}.\ \cite{bercu97} it was proved that the Yule-Walker
estimator satisfies an LDP with rate function given by
  \begin{equation*}
   S(c) = \begin{cases}\displaystyle
           \frac{1}{2} \log\left(\frac{1 + \theta^2 - 2 \theta c}{1 -
c^2}\right),& \mbox{if } |c| < 1,\\
           \infty,& \mbox{otherwise}.
          \end{cases}
  \end{equation*}
Latter on, Bercu {\it et al}.\ \cite{bercu00} provided a Sharp
Large Deviation Principle (SLDP) for Hermitian quadratic forms of stationary
Gaussian processes, obtaining the Yule-Walker's SLDP as a particular case.
In Bercu \cite{bercu01}, the study on LDP of the Yule-Walker estimator
in AR(1) processes was extended to the unstable ($|\theta| = 1$) and explosive
($|\theta| > 1$) cases.

Here we obtain the result from Bercu {\it et al}.\ \cite{bercu97} by using
Proposition \ref{SupConj} and the Contraction Principle. Since the rate
function can be related to the sequence of probabilities
$\mathbb{P}\left(\tilde{\theta}_n \geq c\right)$, for $|\tilde{\theta}_n| < 1$
and $n \geq 2$, it
actually makes sense to get $S(c)$ finite, for $|c|<1$, and infinite
when $|c| \geq 1$.

  From \eqref{thetan}, note that
  \begin{equation*}
   \tilde{\theta}_n = \frac{\sum_{k=2}^n X_k X_{k-1}}{\sum_{k=1}^n X_k^2} =
f(\W_n),
  \end{equation*}
  where $\W_n$ is the random vector given in \eqref{Wn} and
$f: \R^2 \rightarrow \R$ is the continuous function defined by
  \begin{equation}\label{thef}
   f(x, y) = \frac{y}{x}, \quad \mbox{for } 0 < x \mbox{ and } |y| < x.
  \end{equation}

Since $(\W_n)_{n \geq 2}$ satisfies an LDP with rate function $J(\cdot,
\cdot)$, given in \eqref{ratear1}, the Contraction Principle
is applicable and $(\tilde{\theta}_n)_{n \geq 2}$ must satisfy an LDP with rate
function, given by $I_{\theta}(\cdot): \R \rightarrow [0,\infty]$.
Then $I_{\theta}(\cdot)$ can be computed from \eqref{Rate2} and
\eqref{thef} as follows.

  \begin{itemize}
   \item By the Contraction Principle,
   \begin{equation}\label{rathe}
    \begin{split}
    I_{\theta}(c)
    & = \inf_{\{0 < x , \ |y| < x\}} \{J(x,y) \,|\, f(x,y) = c\}
    = \inf_{\{0 < x , \ |y| < x\}} \{J(x,y) \,|\, y - c x = 0\} \\
    & = \inf_{0 < x}\left\{\frac{1}{2}\left[x (1 - 2 c \theta +\theta^2) -
1 + \log\left(\frac{1}{x(1 - c^2)}\right) \right]\right\}, \quad \mbox{for }
|c|<1;
    \end{split}
   \end{equation}
   \item The infimum in \eqref{rathe} is attained at
  \begin{equation}\label{xctheta}
   x_c = \frac{1}{1 - 2 c \theta + \theta^2};
  \end{equation}

  \item  Inserting \eqref{xctheta} into \eqref{rathe},
$I_{\theta}(c)$, for $|c|<1$, reduces itself to
  \begin{equation*}
   I_{\theta}(c) = \frac{1}{2}\left[x_c (1 - 2 c \theta +\theta^2) - 1 +
\log\left(\frac{1}{x_c(1 - c^2)}\right)\right] = \frac{1}{2} \log\left(\frac{1 +
\theta^2 - 2 \theta c}{1 - c^2}\right);
  \end{equation*}

  \item Considering that $I_{\theta}(c) = \infty$, for $|c| \geq 1$, we obtain
$I_{\theta}(c) = S(c), \forall c \in \R$.
  \end{itemize}

  Therefore, we get the same result as in expression (4.6) in Bercu {\it et
al}.\ \cite{bercu97}.
The graph of $I_{\theta}(\cdot)$ is illustrated in Figure \ref{Itheta} for
three different values of $\theta$.

 \begin{figure}
   \centering
   \includegraphics[scale = 0.4]{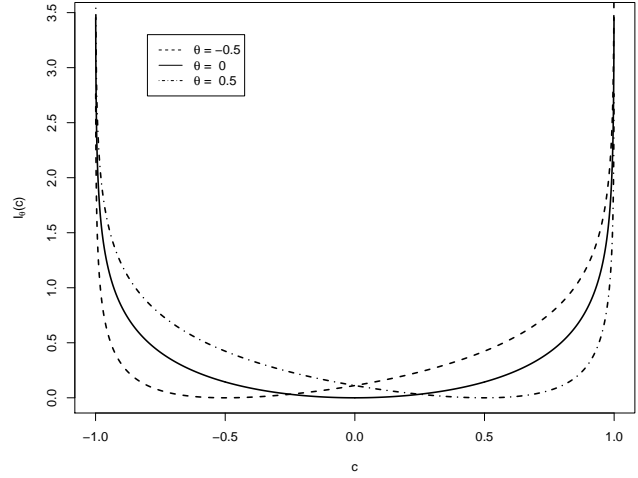}\caption{Graph of
$I_{\theta}(\cdot)$ for $\theta \in \{-0.5, 0, 0.5\}$ and $c \in (-1, 1)$.
\label{Itheta}}
  \end{figure}


\section{Large deviations for the bivariate SQ-Mean}

After finding the rate function for the random sequence $(n^{-1} \sum_{k=1}^n
X_k^2)_{n \geq 2}$ in Section 3.1, there exists a simple variation of that
approach leading to the LDP for the sequence of bivariate SQ-Mean
$(\boldsymbol{S}_n)_{n \in \N}$,
where
\begin{equation}\label{bmSn}
 \boldsymbol{S}_n = \frac{1}{n} \left(\sum_{k=1}^n X_k, \sum_{k=1}^n
X_k^2\right).
\end{equation}
We shall use a result proved in Bryc and Dembo \cite{bryc97}, which we
enunciate below for completeness.

\begin{prop}\label{propbryc}
 Let $(X_n)_{n \in \N}$ be a real-valued centered stationary Gaussian process
whose spectral density $f(\cdot)$ is differentiable. Then,
$(\boldsymbol{S}_n)_{n
\in
\N}$, for $\boldsymbol{S}_n$ given in \eqref{bmSn}, satisfies an LDP (in $\R^2$)
with
good rate function
\begin{equation}\label{rateK}
 K(x, y) = I(y-x^2) + \frac{x^2}{2 f(0)},
\end{equation}
where $0/0:=0$ in \eqref{rateK} and $I(\cdot)$ is the rate function associated
to $(n^{-1} \sum_{k=1}^n X_k^2)$.
\end{prop}
\begin{proof}
 See section 3.5 in Bryc and Dembo \cite{bryc97}.
\end{proof}

We dedicate the next two subsections to the particular study of the LDP of the
bivariate SQ-Mean when $(X_n)_{n \in \N}$ is an AR(1) process (Subsection
\ref{SQar1}) and is an MA(1) process (Subsection 4.2). Since
the LDP for the Quadratic Mean is already available for the AR(1) process, it
is easy to show such property in this case. For the MA(1) process,
however, we must first derive the LDP of the Quadratic Mean in order to apply
Proposition \ref{propbryc} and to provide the LDP for the bivariate SQ-Mean,
likewise.

\subsection{AR(1) process}\label{SQar1}

Since the AR(1) process $(X_n)_{n \in \N}$ in \eqref{ar1} is a
real-valued centered stationary Gaussian process, it follows from Proposition
\ref{propbryc} that $(\boldsymbol{S}_n)_{n \in \N}$ satisfies an LDP
with rate function
\begin{equation*}
 J_{\boldsymbol{S}}(x,y) = \mathbb{I}(y-x^2)+\frac{x^2}{2 g_\theta(0)},
\end{equation*}
where $\mathbb{I}(\cdot)$ is defined by \eqref{IBryc} and
$g_\theta(\cdot)$ denotes the spectral density function given in
\eqref{spectral}. Note that $g_\theta(\cdot)$ is differentiable. The explicit
rate function is given by
\begin{equation*}
J_{\boldsymbol{S}}(x,y) =
  \begin{cases}
             \frac{1}{2}\left[y(1+\theta^2)-2 x^2 \theta  -
\sqrt{1+4\theta^2(y-x^2)^2} -
\log\left(\frac{2(y-x^2)}{1+\sqrt{1+4\theta^2(y-x^2)^2}}\right)\right],&
\mbox{if } y > x^2,\\
             \infty,& \mbox{if } y \leq x^2.
            \end{cases}
\end{equation*}

As a consequence, by an application of the Contraction Principle with the
auxiliary continuous
function $f_1(x,y) = x$, we are able to obtain the rate function for the AR(1)
Sample Mean $\overline{X}_n = n^{-1}\sum_{k=1}^n X_k$. Following the same steps
from
Section
\ref{ldpQM}, notice that the infimum
\begin{equation*}
\begin{split}
 I_{\overline{X}}(c) &= \inf_{y > x^2}\{J_{\boldsymbol{S}}(x,y)\,|\,f_1(x,y)=c\}
=
\inf_{y > c^2} J_{\boldsymbol{S}}(c,y) \\
&= \inf_{y > c^2} \left\{\frac{1}{2}\left[y(1+\theta^2)-2 c^2 \theta  -
\sqrt{1+4\theta^2(y - c^2)^2} - \log\left(\frac{2(y - c^2)}{1 + \sqrt{1 +
4\theta^2(y - c^2)^2}}\right)\right]\right\}
\end{split}
\end{equation*}
is attained at
\begin{equation*}
 y_c = \frac{1 +c^2(1 - \theta ^2)}{1-\theta ^2}.
\end{equation*}
Hence, the sequence $(\overline{X}_n)_{n \in \N}$
satisfies an LDP with rate function
\begin{equation*}
 I_{\overline{X}}(c) = J_{\boldsymbol{S}}(c,y_c) = \frac{c^2 (1 - \theta)^2}{2},
\quad \mbox{for } c \in \R.
\end{equation*}
The graphs of $I_{\overline{X}}(\cdot)$ are depicted in Figure \ref{IS} for
three different values of $\theta$. Notice that, $I_{\overline{X}}(\cdot)$ has
the shape of a parabola.

\begin{figure}
   \centering
   \includegraphics[scale = 0.4]{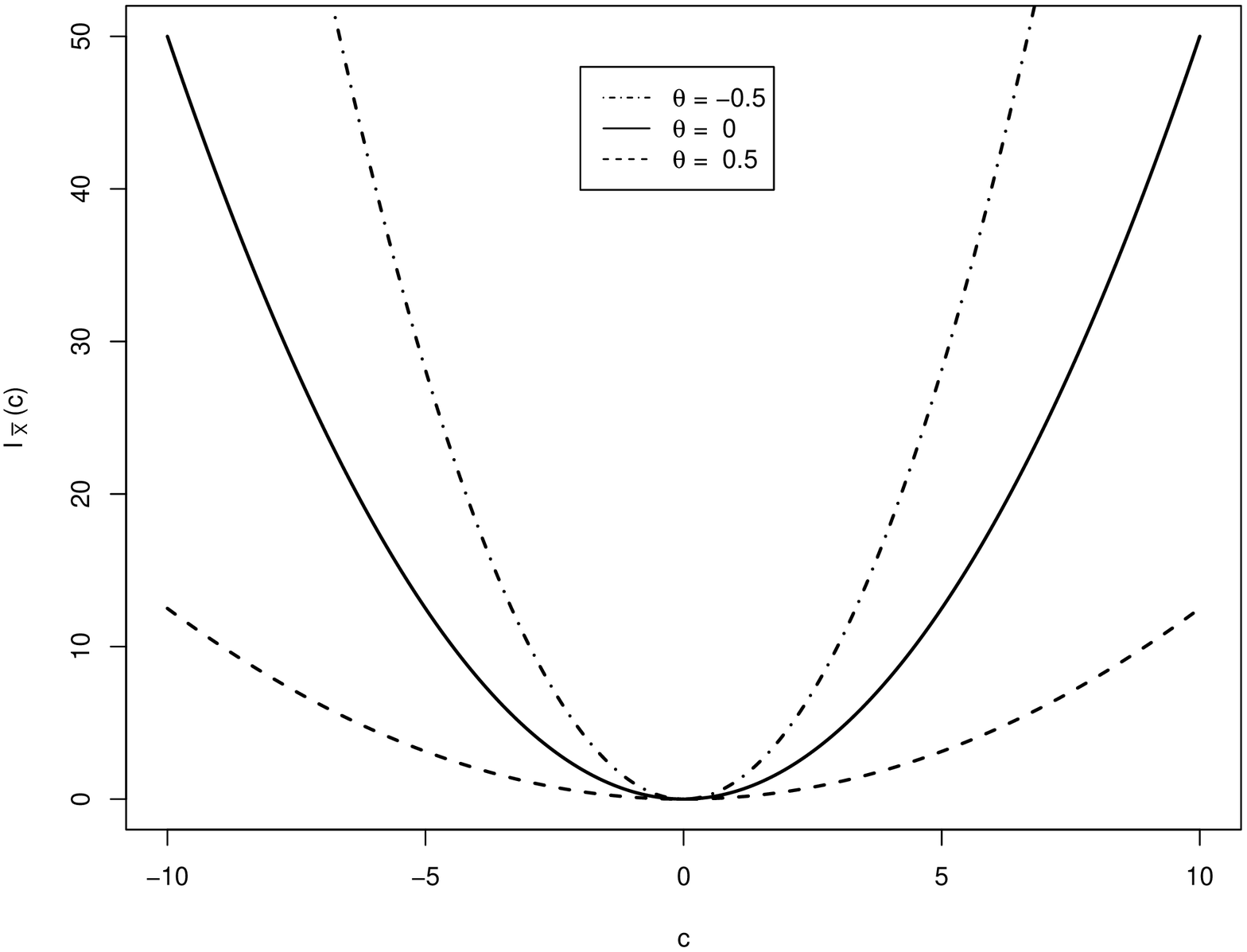}\caption{Graph of
$I_{\overline{X}}(\cdot)$ for $\theta \in \{-0.5, 0, 0.5\}$ and $c \in [-10,
10]$. \label{IS}}
  \end{figure}

\subsection{MA(1) process}

Consider the MA(1) process, defined by the equation
\begin{equation}\label{ma1}
 Y_{n} = \varepsilon_{n} + \phi \varepsilon_{n-1}, \qquad \mbox{with }
|\phi|<1 \mbox{ and } n \in \N.
\end{equation}
 Here, we assume that the innovations $(\varepsilon_n)_{n \geq 0}$ are
i.i.d., with $\varepsilon_n \sim \mathcal{N}(0,1)$. Then, $Y_n \sim
\mathcal{N}(0, 1+\phi^2)$, for each $n \in \N$, and the spectral density
function associated to $(Y_n)_{n \in \N}$ is given by
 \begin{equation*}
   h_\phi(\omega) = 1 + \phi^2 + 2 \phi \cos(\omega), \quad \mbox{for } \omega
\in \T = [-\pi, \pi).
 \end{equation*}
 The process $(Y_n)_{n \in \N}$ is stationary for any $\phi \in
\R$ (see definition 3.4 in Shumway and Stoffer \cite{shumway16}). Nevertheless,
the assumption
$|\phi| < 1$ in \eqref{ma1} ensures that the process is also invertible
and that $h_\phi(\cdot)$ is positive for all $\omega \in \T$.

Let us denote by
 \begin{equation*}
  \tilde\gamma_n(1) = \frac{1}{n} \sum_{k=1}^n Y_k^2
 \end{equation*}
 the Quadratic Mean of a random sample $Y_1, \cdots, Y_n$, following the
MA(1) process described in \eqref{ma1}. Since the autocovariance
function of
$(Y_n)_{n \in \N}$ is equal to
\begin{equation*}
 \gamma_Y(k) = \begin{cases}
                1+\phi^2,& \mbox{if } k = 0,\\
                \phi,& \mbox{if } |k| = 1,\\
                0,& \mbox{if } |k| > 1,
               \end{cases}
\end{equation*}
it is known (see section 7.3 in Brockwell and Davis \cite{brockwell91}) that
\begin{equation*}
 \tilde{\gamma}_n(0) \xrightarrow{n \to \infty} \gamma_Y(0) = 1 + \phi^2, \quad
\mbox{almost surely}.
\end{equation*}

We shall prove that the sequence $(\tilde\gamma_n(0))_{n \in \N}$ satisfies
an LDP. For this reason, consider the normalized cumulant generating function
\begin{equation*}
 L_n(\lambda) = \frac{1}{n} \log \E(e^{\lambda \tilde\gamma_n(0)}).
\end{equation*}
In this case, the asymptotic distribution of $L_n(\cdot)$ is
known (see Grenander and Szeg\"{o} \cite{grenander58}) and we immediately
obtain the convergence
\begin{equation*}
 \lim_{n \to \infty} L_n(\lambda) = L(\lambda) =
 \begin{cases}
  -\frac{1}{4 \pi} \int_\T \log[1 - 2 \lambda h_\phi(\omega)]\, d\omega,&
\mbox{if } \lambda \in \left(-\infty, \frac{1}{2 M_{h_\phi}}\right), \\
  \infty,& \mbox{otherwise},
 \end{cases}
\end{equation*}
 where $M_{h_\phi}$ denotes the essential suppremum of $h_{\phi}(\cdot)$,
given by
\begin{equation*}
 M_{h_\phi} = \begin{cases}
        \frac{1}{2(1+\phi)^2},& \mbox{if } \phi \geq 0,\\
        \frac{1}{2(1-\phi)^2},& \mbox{if } \phi < 0.
       \end{cases}
\end{equation*}

As presented in Bercu {\it et al}.\ \cite{bercu97} and corollary 1 in Bryc and
Dembo \cite{bryc97},
$(\tilde\gamma_n(0))_{n \in \N}$ satisfies an LDP whose good rate function is
the Fenchel-Legendre dual of $L(\cdot)$, given by
\begin{equation}\label{supMA1}
 K_\phi(x) =
 \begin{cases}
 \sup_{\lambda < \frac{1}{2 M_{h_\phi}}} \left\{x \lambda + \frac{1}{4 \pi}
\int_\T \log[1-2 \lambda h_{\phi}(\omega)]\, d\omega\right\},& \mbox{for } x >
0,\\
 \infty,& \mbox{for } x  \leq 0.
 \end{cases}
\end{equation}
Since
\begin{equation*}
 \begin{split}
  \int_\T \log[1-2 \lambda h_{\phi}(\omega)]\, d\omega &= \int_\T \log[1-2
\lambda (1 + \phi^2) - 4 \lambda \phi \cos(\omega)]\,d\omega \\
 & = \log \left[\frac{1-2\lambda(1+\phi^2)+\sqrt{(1-2\lambda(1+\phi^2))^2 -
16 \lambda^2 \phi^2}}{2}\right],
 \end{split}
\end{equation*}
the suppremum in \eqref{supMA1} is attained at
\begin{equation}\label{lambdaphi}
 \lambda_\phi(x)=\frac{A_\phi(x) +\frac{B_\phi(x)}{C_\phi(x)} +
C_\phi(x)}{12 x^2 \left(\phi ^2-1\right)^2},
\end{equation}
where
\begin{equation*}
 A_\phi(x) = 4 x \left(x \left(\phi ^2+1\right)-\left(\phi^2
-1\right)^2\right),
\end{equation*}
\begin{equation*}
 B_\phi(x) = 4 x^2 \left(x^2 \left(\phi ^4+14 \phi
^2+1\right)+4 x \left(\phi ^2+1\right) \left(\phi ^2-1\right)^2+\left(\phi
^2-1\right)^4\right),
\end{equation*}
and
\begin{equation*}
 \begin{split}
 C_\phi(x) =& -(1+i\sqrt{3})\left[-x^6 \left(\phi ^6-33 \phi ^4-33 \phi
 ^2+1\right)-6 x^5 \left(\phi ^2-1\right)^2 \left(\phi ^4-10 \phi
 ^2+1\right)\right.\\
 &\left. \hspace{24mm} +6 x^4 \left(\phi ^2-1\right)^4 \left(\phi
^2+1\right)+x^3
 \left(\phi ^2-1\right)^6+3 \sqrt{3}\sqrt{c_\phi(x)} \ \right]^{1/3},
 \end{split}
\end{equation*}
with
\begin{equation*}
\begin{split}
 c_\phi(x) = -x^8 \left(\phi ^2-1\right)^4 &\left( 4 x^4 \phi ^2+32 x^3
\left(\phi ^4+\phi ^2\right)+x^2 \left(\phi ^4+46 \phi ^2+1\right) \left(\phi
^2-1\right)^2\right.\\
& \left. \ +6 x \left(\phi ^2+1\right) \left(\phi^2-1\right)^4+\left(\phi
^2-1\right)^6\right).
\end{split}
\end{equation*}

\begin{remark}
  Although $C_\phi(\cdot)$ appears in a complex form, it can be proved that
$B_\phi(x)/C_\phi(x) + C_\phi(x) \in \R$, for any $x >
0$. In fact, $\lambda_\phi(x)$ in \eqref{lambdaphi} is one of the
solutions from the polynomial equation
\begin{equation*}
 \begin{split}
& \lambda^3 \left(4 x^2 \phi ^4-8 x^2 \phi ^2+4
x^2\right)+\lambda^2 \left(-4 x^2 \phi ^2-4 x^2+4 x \phi ^4-8 x \phi ^2+4
x\right)\\
&+\lambda \left(x^2-4 x \phi ^2-4 x+\phi ^4-2 \phi ^2+1\right)+x-\phi ^2-1 = 0,
\end{split}
\end{equation*}
which has three real roots if $x > 0$. Moreover, we have $\lambda_\phi(x) <
\frac{1}{2 M_{h_\phi}}.$
\end{remark}

We conclude that $(\tilde\gamma_n(0))_{n \in \N}$ satisfies an LDP with
rate function given by
\begin{equation}\label{Kphi}
 \begin{split}
 &K_\phi(x) = x \lambda_\phi(x) + \frac{1}{2} \log \left[\frac{1 -
2\lambda_\phi(x)(1+\phi^2) + \sqrt{(1- 2 \lambda_\phi(x) (1+\phi^2))^2 - 16
\lambda_\phi(x)^2 \phi^2}}{2}\right] \\
&= \frac{f_\phi(x)}{12 x \left(\phi
^2-1\right)^2} + \frac{1}{2} \log \left(\frac{1}{2}-\frac{\left(\phi
^2+1\right) f_\phi(x)}{12 x^2 \left(\phi ^2-1\right)^2}+
\sqrt{\left(\frac{1}{2}-\frac{\left(\phi ^2+1\right) f_\phi(x)}{12
x^2 \left(\phi ^2-1\right)^2}\right)^2-\frac{\phi ^2
f_\phi(x)^2}{36 x^4\left(\phi ^2-1\right)^4}}\right),
\end{split}
\end{equation}
for all $x > 0$ and $K_\phi(x) = \infty$, for $x \leq 0$, with
\begin{equation*}
 f_\phi(x) = A_\phi(x) + \frac{B_\phi(x)}{C_\phi(x)} + C_\phi(x).
\end{equation*}
The graph of $K_\phi(\cdot)$ is illustrated in Figure \ref{K1} for four
different values of $\phi$ and $x \in (0, 5]$.

 \begin{figure}[ht]
   \centering
   \includegraphics[scale = 0.4]{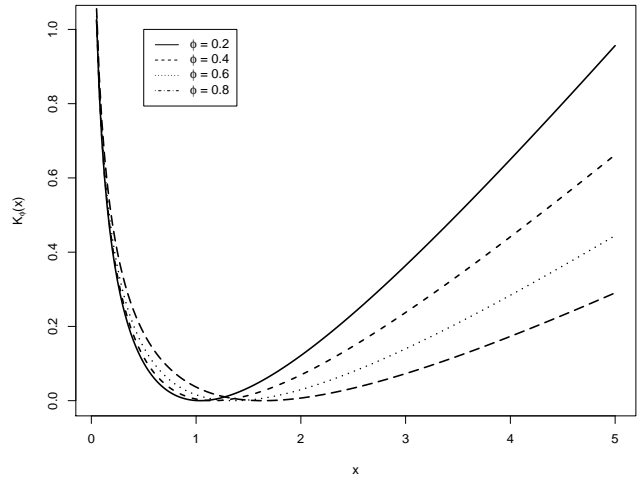}\caption{Graphs of
$K_\phi(\cdot)$ for $\phi \in \{0.2, \, 0.4, \, 0.6, \, 0.8\}$ and $x \in (0,
5]$. \label{K1}}
  \end{figure}

  By Proposition \ref{propbryc} we may now conclude that
$(\boldsymbol{S}_n)_{n\in \N}$ satisfies an LDP with rate function
\begin{equation*}
 K_{\boldsymbol{S}}(x,y) =
 \begin{cases}
  K_\phi(y-x^2)+\frac{x^2}{2(1 + \phi)^2},& \mbox{if } y > x^2,\\
  \infty,& \mbox{if } y \leq x^2.
 \end{cases}
\end{equation*}
where $K_\phi(\cdot)$ is given in \eqref{Kphi}.

Note that, by the Contraction Principle, the sequence
$(f_1(\boldsymbol{S_n}))_{n
\in \N} = \left(n^{-1} \sum_{k=1}^{n} Y_k\right)_{n \in \N}$, where $f_1(x,y) =
x$ and $\boldsymbol{S}_n = n^{-1} \left(\sum_{k=1}^n Y_k, \sum_{k=1}^n
Y_k^2\right)$,
must
satisfy an LDP with rate function
\begin{equation}\label{nontrivinf}
 \begin{split}
 I_{\overline{Y}}(c) &= \inf_{y > x^2}\{K_{\boldsymbol{S}}(x,y)\,|\,f_1(x,y)=c\}
\\
& = \inf_{y > c^2} K_{\boldsymbol{S}}(c,y) = \inf_{y > c^2}
\left\{K_\phi(y-x^2)+\frac{x^2}{2(1 +\phi)^2}\right\}.
 \end{split}
\end{equation}
However, when trying to compute the infimum in \eqref{nontrivinf}, we face
a non-trivial problem.

Fortunately, an LDP for the Sample Mean of moving average processes has already
been given in Burton and Dehling \cite{burton90}, where the authors considered
the sequence
\begin{equation*}
 X_n = \sum_{k \in \Z} a_{k+n} \, \varepsilon_k, \quad \mbox{for } n \in \Z,
\end{equation*}
 with $(\varepsilon_n)_{n \in \Z}$ a sequence of i.i.d.\ random variables.
They proved the LDP under the hypotheses that $(a_n)_{n \in \Z}$ is an
absolutely summable sequence and that the moment generating function
$\E(e^{t\varepsilon_1})$ is finite, for all $t \in \R$. In
Djellout and Guillin \cite{djellout01}, a similar approach has been given. In
this paper, the authors proved an analogous result under the hypotheses that
the sequence $(\varepsilon_n)_{n \in Z}$ is bounded and that $\sum_{k \in \Z}
a_k^2 < \infty$. If we set $a_{2n} = 1$, $a_{2n-1} = \phi$ and $a_k = 0$ for $k
\in \Z
\setminus \{2n,2n-1\}$, then $X_n = \varepsilon_n + \phi \varepsilon_{n-1}$ is
the MA(1) process given in \eqref{ma1}, as long as the same hypotheses for
the distribution of $(\varepsilon_n)_{n \geq 0}$ are considered. Then by theorem
2.1 in Burton and Dehling \cite{burton90}, the Sample Mean $(\overline{Y}_n)_{n
\in \N} =
\left(n^{-1} \sum_{k=1}^n Y_k \right)_{n \in \N}$ satisfies an LDP with rate
function
\begin{equation*}
 I_{\overline{Y}}(c) = \sup_{\lambda \in
\R} \left\{\frac{c \, \lambda}{1+\phi} - \frac{\lambda^2}{2}\right\} =
\frac{c^2}{2(1+\phi)^2}, \quad \mbox{for } c \in \R.
\end{equation*}

\section{Conclusion}

In this work, we showed that an LDP is available for the sequence $(\W_n)_{n
\geq 2}$, given in \eqref{Wn}. The same technique to find such LDP is not
restricted to the AR(1) process. There may exist other classes of processes
that can be explored as well. If we take another process $(Z_n)_{n \in \N}$
which still has a multivariate Gaussian distribution, equipped with another
spectral density function, other than the one given in \eqref{spectral},
the proposed technique may remain valid. The LDP is, however, not always
guaranteed and in most cases, the rate function is hard to compute. This
difficulty mainly arises when trying to compute a closed form for the
Fenchel-Legendre transform. Besides that, to obtain a similar convergence
result as given in Lemma \ref{lemma2}, for another class of Gaussian
processes remains an intriguing problem. A remarkable class of processes that
requires a more sophisticated approach, is the class of MA(1) process, which
was not covered in this work when evaluating the LDP for the random vectors
$(\W_n)_{n \geq 2}$.

In Section 3, we presented three important particular examples using the
previous reasoning from Section 2, together with the Contraction Principle. Two
of these examples were already known from Bercu {\it et al}.\ \cite{bercu97}
and Bryc and Smolenski \cite{bryc93} for
univariate sequences. Here we obtained them as a continuous transformation of
the random vector $\W_n$, given in \eqref{Wn}. In Subsection 3.2, we presented
a result which we believe is new in the literature. In Subsection 3.3, the LDP
for the Yule-Walker estimator was obtained, via the Contraction Principle,
getting the same result as in Bercu {\it et al}.\ \cite{bercu97}. The approach
used here, first
proving an LDP for bivariate random vectors and then particularizing to
univariate random sequences via Contraction Principle has recently been used
with continuous stochastic processes by Bercu and Richou \cite{bercu15}, where
the authors investigated the LDP of the maximum likelihood estimates for the
Ornstein-Uhlenbeck process with shift. A similar approach was subsequently used
by the same authors in Bercu and Richou \cite{bercu17}, allowing them to
circumvent the classical
difficulty of non-steepness.

In Section 4, we provided an LDP for the sequence of bivariate SQ-Mean, for both
AR(1) and MA(1) processes. For the AR(1) process, the computations were simple
and the previous technique of proving an LDP for the bivariate random vector
$\W_n$ was extremely helpful. Nevertheless, when dealing with the MA(1)
process, we found some issues due to the complexity of the computations
involved. The same technique explored here may perhaps be available for general
AR($d$) processes with Gaussian innovations. This is an important issue to be
explored in the future.

\appendix

\section{Proof of Lemma \ref{domD}}\label{appenA}
  In this appendix, we give the details for the proof of Lemma \ref{domD},
which was based on the techniques given in page 270 in Jensen
\cite{jensen95}. In summary, we use Sylvester's Criterion (see theorem 7.2.5
in Horn \cite{horn13}) to check for the positive definiteness of each principal
minor of $D_{n,\boldsymbol{\lambda}}$ and resort to the use of an auxiliary
function
with its corresponding iterates.

  By Sylvester's Criterion, $D_{n,\boldsymbol{\lambda}}$ is positive definite,
if
and only if, the principal minors of $D_{n,\boldsymbol{\lambda}}$ are positive.
Hence, we analyze each one of the principal minors of
$D_{n,\boldsymbol{\lambda}}$ as follows:

\begin{itemize}
 \item  {\it 1-st Step:} since the first principal minor of
$D_{n,\boldsymbol{\lambda}}$ is $r_1 = 1 - 2 \lambda_1$, we require that $r_1
> 0$. As a consequence, since $p = r_1 + \theta^2$, we obtain
  \begin{equation*}
   0 < r_1 < p \Rightarrow 0 < p.
  \end{equation*}

 \item   {\it 2-nd Step:} the second principal minor of
$D_{n,\boldsymbol{\lambda}}$
is defined as the determinant
  \begin{equation}\label{2nd}
   \left|\begin{array}{cc}
        r_1    & q \\
        q    & p
       \end{array}\right| = p\,r_1 - q^2 = \left(p - \frac{q^2}{r_1}\right) r_1.
  \end{equation}
  Since we already restricted our analysis for $r_1 > 0$, \eqref{2nd}
requires in addition that $r_2:= p - \frac{q^2}{r_1} > 0$.

 \item
  {\it 3-rd Step:} the third principal minor of $D_{n,\boldsymbol{\lambda}}$ is
the
determinant
  \begin{equation}\label{3rd}
   \left|\begin{array}{ccc}
        r_1 & q & 0   \\
        q & p & q \\
        0   & q & p \\
       \end{array}\right| = p^2\,r_1 - q^2\, r_1 - q^2\,p = \left(p -
\frac{q^2}{p - \frac{q^2}{r_1}}\right)\left(p - \frac{q^2}{r_1}\right)\,r_1.
  \end{equation}
  Since we already restricted our analysis for $r_1 > 0$ and $p -
\frac{q^2}{r_1} > 0$, \eqref{3rd} requires that $r_3:=\left(p -
\frac{q^2}{p - \frac{q^2}{r_1}}\right) > 0$.

 \item   {\it k-th Step:} by induction, the $k$-th principal minor of
$D_{n,\boldsymbol{\lambda}}$, for $ 1 \leq k \leq n-1$, is the determinant
  \begin{equation*}
   \left|\begin{array}{ccccc}
        r_1    & q    & 0 & \cdots & 0 \\
        q    & p    & \ddots & \ddots & \vdots \\
        0      & q & \ddots & q & 0\\
        \vdots & \ddots & \ddots & p & q\\
        0      & \cdots & 0 & q & r_1 \\
       \end{array}\right| = r_k \, r_{k-1} \, \cdots \, r_2 \, r_1,
  \end{equation*}
  for $r_2 = p - \frac{q^2}{r_1}$, $r_3 = \left(p - \frac{q^2}{p -
\frac{q^2}{r_1}}\right)$ and $r_k  = G^{k-1}(r_1)$, where $G^k$ denotes the
$k$-th iterate of $G:(0, \infty) \rightarrow (0, \infty)$, given by
  \begin{equation*}
   G(a) = p - \frac{q^2}{a}.
  \end{equation*}

  Since $n \in \N$ is arbitrary, we must require that $G^k(r_1) > 0$, for all $k
\in \N$. Without loss of generality, we may assume that $q \neq 0$ (if $q = 0$,
then $D_{n,\boldsymbol{\lambda}}$ is a diagonal matrix; this happens if and only
if
$\lambda_2 = -\theta$).

  Notice that $G(\cdot)$ has the following two fixed points
  \begin{equation*}
   R = \frac{1}{2}\left(p - \sqrt{p^2 - 4 q^2}\right) \qquad \mbox{ and } \qquad
Q = \frac{1}{2}\left(p+\sqrt{p^2 - 4 q^2}\right).
  \end{equation*}
The point named $Q$ is an attractor point and the point named $R$ is a repulsor
point, provided that $p^2 > 4 q^2$. If $p^2 = 4q^2$, then $P = Q = p$
is neither an attractor, neither a repulsor point. Let us consider henceforth
$p^2 > 4 q^2$.

Observe that $G(\cdot)$ is an increasing concave function. Therefore, the
problem of knowing when $G^k(r_1) > 0,$ for all $k \in \N$, reduces to knowing
where $r_1 > R$. In one hand, every point greater than $R$ converges
towards $Q$ and since $R > 0$, it follows that
  \begin{equation*}
   G^k(r_1) > R > 0, \quad \mbox{for all } k \in \N.
  \end{equation*}
  On the other hand,
  \begin{equation*}
   \forall \, x < R, \ \exists \, n_0 \in \N; \ G^{n_0}(x) < 0.
  \end{equation*}
  Since $r_1 = 1 - 2 \lambda_1 = p - \theta^2$, we get
  \begin{equation}\label{rR}
   r_1 > R \Leftrightarrow r_1 > \frac{p - \sqrt{p^2 - 4 q^2}}{2}
\Leftrightarrow \sqrt{p^2 - 4 q^2} > p - 2 r_1 = p - 2(p - \theta^2) = 2
\theta^2 - p.
  \end{equation}

  If $p \geq 2\theta^2$, then obviously $r_1> R$, since the right-hand side of
\eqref{rR} is non-positive. But if $p < 2\theta^2$, then
  \begin{equation*}
   r_1 > R \Leftrightarrow p^2 - 4 q^2 > \left(2\theta^2 - p\right)^2
   \Leftrightarrow p^2 - 4q^2 > 4\theta^4 - 4 \theta^2 p +p^2 \Leftrightarrow
\theta^2(p - \theta^2) > q^2.
  \end{equation*}
  Therefore, we obtain the domain $\mathcal{D} = \mathcal{D}_1 \cup
\mathcal{D}_2$, where
  \begin{equation*}
   \begin{split}
   &\mathcal{D}_1 = \{r_1 > 0, \ p^2 > 4 q^2, \ p \geq 2 \theta^2\} \quad
\mbox{ and} \\[1.5mm] &\mathcal{D}_2 = \{r_1 > 0, \ p^2 > 4 q^2, \ p < 2
\theta^2, q^2 < \theta^2
(p - \theta^2)\}.
   \end{split}
  \end{equation*}

  Notice that $r_1 > 0$ is equivalent to $p > \theta^2$. Moreover, from
  \begin{equation*}
   0 > -\left(\theta^2 - \frac{p}{2}\right)^2 = - \theta^4 + 2 \theta^2
\frac{p}{2} - \frac{p^2}{4} = - \theta^4 + \theta^2 p - \frac{p^2}{4} =
\theta^2(p - \theta^2) - \frac{p^2}{4},
  \end{equation*}
  we conclude that $\theta^2(p - \theta^2) < \frac{p^2}{4}$. Hence, if $q^2 <
\theta^2 (p - \theta^2)$, it follows that $4 q^2 < p^2$.  Therefore, if $p, q$
belong to
  \begin{equation*}
   \mathcal{D}_1 = \{p \geq 2 \theta^2, \ p^2 > 4 q^2\} \quad \mbox{or} \quad
\mathcal{D}_2 = \{\theta^2 < p < 2 \theta^2, \ q^2 < \theta^2 (p - \theta^2)\},
  \end{equation*}
 then $G^k(r_1) > 0,$ for all $k \in \N$.

  \item    {\it n-th Step}: last but not least, the $n$-th principal minor
(or determinant) of $D_{n,\boldsymbol{\lambda}}$ is
  \begin{equation*}
   \begin{split}
   |D_{n,\boldsymbol{\lambda}}| &=
   \left|\begin{array}{ccccc}
        r_1    & q    & 0 & \cdots & 0 \\
        q    & p    & \ddots & \ddots & \vdots \\
        0      & q & \ddots & q & 0\\
        \vdots & \ddots & \ddots & p & q\\
        0      & \cdots & 0 & q & r_1 \\
       \end{array}\right| = r_1 (r_{n-1} \, r_{n-2} \cdots r_2\, r_1) - q^2
(r_{n-2} \cdots r_2 \, r_1) \\
       & = \left(r_1 - \frac{q^2}{r_{n-1}}\right) (r_{n-1} \, r_{n-2} \cdots
r_2\, r_1)
        = (G^{n-1}(r_1) - \theta^2) (r_{n-1} \, r_{n-2} \cdots r_2\, r_1).
   \end{split}
  \end{equation*}

  It is not difficult to see that, for $n$ large enough, we eventually obtain
$G^{n-1}(r_1) > \theta^2$. Indeed
  \begin{equation*}
   p > \theta^2 \Rightarrow Q > \theta^2 \Rightarrow \lim_{n \to \infty}
G^n(r_1) = Q > \theta^2,
  \end{equation*}
  so that
  \begin{equation*}
   \exists\, n_0 \in N;  n \geq n_0 \Rightarrow G^n(r_1) > \theta^2.
  \end{equation*}

\end{itemize}

The set $\mathcal{D}_1 \cup \mathcal{D}_2$ is therefore, the closed domain
where all principal minors of $D_{n,\boldsymbol{\lambda}}$ are positive, and
consequently, where the matrix $D_{n,\boldsymbol{\lambda}}$ is positive
definite.
Converting the domains $\mathcal{D}_1$ and $\mathcal{D}_2$ to the $(\lambda_1,
\lambda_2)$ notation, we obtain the desired expressions given by
\eqref{D12}.

\section*{Acknowledgments}

M.J. Karling was supported by Coordena\c{c}\~{a}o de Aperfei\c{c}oamento de
Pessoal
de N\'{i}vel Superior (CAPES)-Brazil and Conselho Nacional
de Desenvolvimento Cient\'{i}fico e Tecnol\'{o}gico (CNPq)-Brazil
(170168/2018-2).
A.O. Lopes' research was partially supported by CNPq-Brazil (304048/2016-0).
S.R.C. Lopes' research was partially supported by CNPq-Brazil (303453/2018-4).
The authors wish to express their sincere thanks to Dr.\ Bernard Bercu for
indicating valuable references from the Large Deviations theory.


\end{document}